\documentclass[a4paper,12pt]{amsart}

\usepackage{amssymb,amsbsy,amsmath,amsfonts,amssymb,amscd}
\usepackage{latexsym}
\usepackage{geometry}

\input xy
\xyoption{all}

\newcommand\ga{\gamma}

\newcommand{\RR}{\ensuremath{\mathbb{R}}}
\newcommand{\ZZ}{\ensuremath{\mathbb{Z}}}
\newcommand{\QQ}{\ensuremath{\mathbb{Q}}}

\newcommand{\NN}{\ensuremath{\mathbb{N}}}
\newcommand{\hol}{\ensuremath{\mathcal{O}}}

\newcommand{\PP}{\ensuremath{\mathbb{P}}}

\newcommand{\ra}{\ensuremath{\rightarrow}}

\def\eea{\end{eqnarray*}}
\def\bea{\begin{eqnarray*}}

\newcommand\dual{\mathrel{\raise3pt\hbox{$\underline{\mathrm{\thinspace d
\thinspace}}$}}}
\newcommand\qe{\ifhmode\unskip\nobreak\fi\quad $\Box$}       % box for QED

\def\BOX{\hfill\lower.5\baselineskip\hbox{$\Box$}}
\newtheorem{theo}{Theorem}[section]
\newtheorem{remarkk}[theo]{Remark}
\newenvironment{rem}{\begin{remarkk}\rm}{\end{remarkk}}

\newtheorem{defin}[theo]{Definition}
\newenvironment{definition}{\begin{defin}\rm}{\end{defin}}

\newtheorem{prop}[theo] {Proposition}
\newtheorem{cor}[theo]{Corollary}
\newtheorem{lemma}[theo]{Lemma}
\newtheorem{example}[theo]{Example}

\newcommand{\BT}{\ensuremath{\mathbb{T}}}

\DeclareMathOperator{\Aut}{Aut}
\DeclareMathOperator{\Sing}{Sing}
\DeclareMathOperator{\lcm}{lcm}

\DeclareMathOperator{\Stab}{Stab}

\begin{document}

\title[Surfaces with $p_g=q=0$]{The classification of minimal product-quotient surfaces with
  $p_g=0$.}
\author{I. Bauer,  R. Pignatelli}

\thanks{The present work took place in the realm of the DFG Forschergruppe 790
"Classification of algebraic surfaces and compact complex manifolds",
in particular  the
visit of the second author to Bayreuth was supported by the DFG. The second author is a 
member of G.N.S.A.G.A. of I.N.d.A.M.
We are very grateful to Fritz Grunewald from whom we learnt a lot about group theory, 
mathematics and life. Fritz passed away on March 21, 2010: we lost a very close friend, 
a great mathematician and a wonderful person.
The second author is also indebted to D. Frapporti for pointing out an error in 
a previous version of the program.}

\date{\today}
\maketitle
{\bf This article is dedicated to the memory of our dear friend and collaborator Fritz Grunewald}
\section*{Introduction}

The present article is the fourth in a series of papers (cf. \cite{bacat}, \cite{pg=q=0}, \cite{4names}), where the 
goal is 
to contribute to the classification problem of surfaces of general type by giving a systematic way to 
construct and distinguish
algebraic surfaces. 

We will use the basic notations from the classification theory of complex projective surfaces, in particular
the basic numerical invariants $K_S^2$, $p_g :=h^0(S, \Omega^2_S)$, $q(S) :=h^1(S, \hol_S)$; 
the reader unfamiliar with these
may consult  e.g. \cite{beau}. 

The methods we introduced in the above cited articles, and substantially develop and refine 
in the present paper are in principle applicable to many more situations.
Still we restrict ourselves to the case of surfaces of general type with geometric genus 
$p_g  = 0$.

It is nowadays
well known that minimal surfaces of general type with $p_g = 0$
 yield  a
finite number of irreducible components of the
moduli space of surfaces of general type. Although it is theoretically possible to describe all 
irreducible components of the moduli space corresponding to surfaces of general type with $p_g=0$, this ultimate goal is 
far out of reach, even if there has been a substantial progress in the study of these surfaces especially in the last five years. 
We refer to
to \cite{4names} and \cite{surveypg=0} for a historical account and recent update on what is known 
about surfaces of general type with $p_g = 0$.

We study
the following situation:
let $G$ be a finite group acting on two compact Riemann surfaces $C_1$, $C_2$ of respective genera at 
least 2.
We shall consider the diagonal action of $G$ on $C_1 \times C_2$ and
in this situation we say for short: the action of $G$ on
$C_1 \times C_2$ is {\em unmixed}. By \cite{FabIso} we may assume wlog
that $G$ acts faithfully on both factors.
\begin{defin}\label{prodquot}
The minimal resolution $S$ of the singularities of $X = (C_1 \times C_2)/ G$, where $G$ is a finite group with an unmixed
action on the direct product of two compact Riemann surfaces $C_1$, $C_2$ of respective genera at least
two, is called a {\em product-quotient surface}.

$X$ is called  {\em the quotient model} of the product-quotient surface.
\end{defin}

\begin{rem}
1) It is possible that two product-quotient surfaces with different quotient models are
birational or even isomorphic. By a slight abuse of notation, we still call 
$X= (C_1 \times C_2)/ G$ "the" quotient model of 
the product-quotient
surface, because we use this notation only if we have fixed two Riemann surfaces and 
the action of a finite group on each of them. 

2) If $X$ has "mild" (to be precise: at most canonical) singularities, then the 
quotient model $X$ is equal 
to the canonical model of the 
product-quotient surface $S$, which is unique. For the definition of canonical surface singularities (which are also called rational double points) we refer to \cite{matsuki}, def. 4-2-1, thm. 4-6-7.

3)
In the general setting, i.e., if $X$ has non canonical singularities, $S$ is not  
necessarily a minimal surface, i.e., it may contain smooth rational curves with selfintersection $-1$. 
This is a substantial obstacle we have to overcome in the present paper. We profit from the fact that the 
construction of
our surfaces is quite explicit (cf. section \ref{fakegodeaux}). 

In general there is no 
way to determine rational curves on surfaces of general type. In fact, a famous still unsolved
conjecture by S. Lang asserts that a surface of general type can contain only
 a finite number of rational curves.

\end{rem}

The systematic classification of product-quotient surfaces  with $p_g = 0$ was started and 
carried through in \cite{bacat}, \cite{pg=q=0}, \cite{4names} for all surfaces whose canonical model 
is equal to $(C_1 \times C_2)/G$. 

\cite{bacat} classifies the surfaces $X=(C_1 \times C_2)/G$ with $G$ being an abelian group acting 
freely and $p_g(X)=0$.
This classification is extended in \cite{pg=q=0} to the case of an arbitrary group $G$.
We want to point out that in the first paper all calculations were done by hand, whereas in the 
second one the computations could not be  done by hand, but they were still 
"computer aided hand calculations".

In \cite{4names} instead we dropped the assumption that $G$ acts freely on $C_1 \times C_2$, 
we classified
product-quotient surfaces with $p_g=0$ whose quotient model is indeed the canonical model (i.e.,
$X$ has at most canonical singularities, note that here automatically $K_S^2 >0$). In this case for 
the first time a systematic use of a computer algebra program was strictly needed in order to obtain a 
complete classification.

In the present paper we drop any restriction on the singularities of $X$. We succeed to give a complete
classification of product-quotient surfaces $S$ with $K_S^2 > 0$. 

 In order to obtain  this result we had to 
substantially refine our previous MAGMA code, and for the first time we encountered serious problems
of complexity and memory usage. Especially, as $K_S^2$ gets smaller ($\leq0$), the problem of
finding the possible singular locus of $X$ gets more and more time and memory demanding. In order 
to finish the classification of product quotient surfaces with $p_g = 0$ one has to deal not only with
the above mentioned computational problems, but also
with the problem of bounding the number of rational curves on $S$, which in view of the previously 
mentioned
Lang's conjecture is foreseen to be hard.

We are interested in the minimal model of the constructed surfaces, 
in order to locate them in the geography of the fine classification of the surfaces of general type. 
We determined the minimal model of all 
these surfaces; the last two sections are dedicated to this scope. 
It turns out that all except one are in fact minimal. 
We call this last surface {\it the fake Godeaux surface}, because a minimal surface with the 
same invariants $p_g$ and $K^2$ is called {\it numerical Godeaux surfaces}.
The section \ref{fakegodeaux} is dedicated to it.

 The following summarizes the results of the series of four papers.

\begin{theo}\label{classiso}
\begin{enumerate}
\item Surfaces $S$ isogenous to a product (i.e., $S$ is an \'etale
quotient of a product of two compact Riemann surfaces of respective genera
at least 2 by a finite group) with $p_g(S) = q(S) = 0$ form
17 irreducible connected components of the moduli space
$\mathfrak{M}$ of surfaces of general type. Exactly $13$ of these families are families of
product-quotient surfaces.

\item Minimal product-quotient surfaces with  $p_g=0$ form exactly 72
irreducible families.

\item There is exactly one product-quotient surface with $K_S^2>0$
which is non minimal. It has $K_S^2=1$, $\pi_1(S)=\ZZ/6\ZZ$ and its minimal model has $K^2=3$.
\end{enumerate}

\end{theo}

\begin{rem}
1) Part 1 is proved in \cite{bacat}, \cite{pg=q=0}. 

2) Of the 72 families of part 2, 40 are constructed in \cite{bacat}, \cite{pg=q=0}, \cite{4names}.
The remaining 32 families, as well as the fake Godeaux surface are new, and come out from our 
main classification result here. 

Therefore we contribute to the existing
knowledge about the complex projective surfaces $S$ of general type with
$p_g(S) = 0$ and their moduli spaces, constructing 33 new families
of such surfaces realizing 14 hitherto unknown topological types.

\end{rem}

\noindent
The product-quotient surfaces mentioned in part 2 of the above theorem are listed in 
tables \ref{K2>4} and \ref{K2<3}. We list the following information in the columns of the tables:
\begin{itemize}
\item $\Sing X$ is given as a sequence of rational numbers with multiplicities, describing the types of the cyclic 
quotient singularities, e.g., $2/3^2$ means 2 singular points of type $\frac13 (1,2)$;
\item $N$ is the number of irreducible families; indeed our tables have only 60 lines, but we collect in the same line $N$ families, which share all the other data; the number of lines, counted with  multiplicity $N$ is 72 (the number of families of theorem \ref{classiso}, 2));
\item $K_S^2$ is the selfintersection of the canonical divisor, $G$ the group, $H_1$ is the homology, and 
$\pi_1$ is the fundamental group.
\item $t_1, t_2$ are the signatures of the corresponding polygonal groups, cf. definition \ref{polgr} 
and the subsequent discussion.
\end{itemize}

For the groups occuring in  tables \ref{K2>4}, \ref{K2<3}, we use the following notation: we denote by $\ZZ_d$  the cyclic group of order $d$, 
 $\mathfrak{S}_n$ is the symmetric group in $n$
letters, $\mathfrak{A}_n$  is the alternating group and $Q_8$ is the quaternion group of order 8.

\noindent
$PSL(2,7)$ is the group of $2 \times 2$ matrices over
${\mathbb F}_7$ with determinant $1$ modulo the subgroup generated by $-Id$.

\noindent
$D_{p,q,r}=\langle x,y|x^p,y^q,xyx^{-1}y^{-r} \rangle$, and
$D_n=D_{2,n,-1}$ is the usual dihedral group of order $2n$.

$G(n,k)$ for instance is the $k$-th group of order $n$ in the MAGMA database of small groups.

\begin{table}
\renewcommand{\arraystretch}{1,3}
\caption{Minimal product-quotient surfaces of general type with $p_g=0$, $K^2\geq
4$}\label{K2>4}
\tiny \begin{tabular}{|c|c|c|c|c|c|c|c|}
\hline
$K^2_S$&Sing X&$t_1$&$t_2$&$G$&N&$H_1(S,{\mathbb Z})$&$\pi_1(S)$\\
\hline\hline 8& $\emptyset$ & $2, 5^2$ &$3^4$ &  ${\mathfrak A}_5$
&$1$ &$\ZZ_3^2 \times
\ZZ_{15}$&$1 \rightarrow \Pi_{21} \times \Pi_{4} \rightarrow \pi_1
\rightarrow G \rightarrow 1$\\ 8&
$\emptyset$ & $5^3$ & $2^3,3$ & ${\mathfrak A}_5$ & $1$
&$\ZZ_{10}^2$&$1 \rightarrow \Pi_{6} \times
\Pi_{13} \rightarrow \pi_1 \rightarrow G \rightarrow 1$\\ 8&
$\emptyset$ & $3^2,5$ & $2^5$&
${\mathfrak A}_5$ & $1$ &$\ZZ_2^3 \times \ZZ_6$&$1 \rightarrow
\Pi_{16} \times \Pi_{5} \rightarrow
\pi_1 \rightarrow G \rightarrow 1$\\ 8& $\emptyset$ & $2,4,6$&  $2^6$
&  ${\mathfrak S}_4\times
\ZZ_2$ &$1$ &$\ZZ_2^4 \times \ZZ_4$&$1 \rightarrow \Pi_{25} \times
\Pi_{3} \rightarrow \pi_1
\rightarrow G \rightarrow 1$\\ 8& $\emptyset$ & $2^2, 4^2$ &$2^3, 4$
&  ${\rm G}(32,27)$  &$1$
&$\ZZ_2^2 \times \ZZ_4 \times \ZZ_8$&$1 \rightarrow \Pi_{5} \times
\Pi_{9} \rightarrow \pi_1
\rightarrow G \rightarrow 1$\\ 8& $\emptyset$ & $5^3$& $5^3$ &
$\ZZ_5^2$ &$2$ &$\ZZ_5^2$&$1
\rightarrow \Pi_{6} \times \Pi_{6} \rightarrow \pi_1 \rightarrow G
\rightarrow 1$\\ 8& $\emptyset$ &
$3,4^2$ & $2^6$ &  ${\mathfrak S}_4$ &$1$ &$\ZZ_2^4 \times \ZZ_8$&$1
\rightarrow \Pi_{13} \times
\Pi_{3} \rightarrow \pi_1 \rightarrow G \rightarrow 1$\\ 8&
$\emptyset$ & $2^2,4^2$&  $2^2,4^2$&
${\rm G}(16,3)$ &$1$ &$\ZZ_2^2 \times \ZZ_4 \times \ZZ_8$&$1
\rightarrow \Pi_{5} \times \Pi_{5}
\rightarrow \pi_1 \rightarrow G \rightarrow 1$\\ 8& $\emptyset$ &
$2^3,4$ & $2^6$ &  ${\rm
D}_4\times\ZZ_2$ &$1$ &$\ZZ_2^3 \times \ZZ_4^2$&$1 \rightarrow
\Pi_{9} \times \Pi_{3} \rightarrow
\pi_1 \rightarrow G \rightarrow 1$\\ 8& $\emptyset$ &$2^5$  & $2^5$&
$\ZZ_2^4$ &$1$
&$\ZZ_2^4$&$1 \rightarrow \Pi_{5} \times \Pi_{5} \rightarrow \pi_1
\rightarrow G \rightarrow 1$\\ 8&
$\emptyset$ & $3^4$ &  $3^4$&  $\ZZ_3^2$ & $1$ &$\ZZ_3^4$&$1
\rightarrow \Pi_{4} \times \Pi_{4}
\rightarrow \pi_1 \rightarrow G \rightarrow 1$\\ 8& $\emptyset$ &
$2^5$ &$2^6$ &  $\ZZ_2^3$ &$1$
&$\ZZ_2^6$&$1 \rightarrow \Pi_{3} \times \Pi_{5} \rightarrow \pi_1
\rightarrow G \rightarrow 1$\\
\hline\hline
   6&$1/2^2$&$2^3, 4$&$2^4, 4$& ${\mathbb Z}_2 \times D_4$ &1&
${\mathbb Z}_2^2 \times {\mathbb
Z}_4^2$ & $1\rightarrow {\mathbb Z}^2 \times \Pi_2 \rightarrow \pi_1
\rightarrow {\mathbb Z}_2^2
\rightarrow 1$\\
   6&$1/2^2$&$2^4,4$ & $2, 4, 6$ &${\mathbb Z}_2 \times {\mathfrak
S}_4$ &1&   ${\mathbb Z}_2^3
\times {\mathbb Z}_4$ &  $1\rightarrow \Pi_2 \rightarrow \pi_1
\rightarrow  {\mathbb Z}_2
\times{\mathbb Z}_4 \rightarrow 1$   \\
   6&$1/2^2$&$2, 5^2$ & $2, 3^3$ & ${\mathfrak A}_5$&1&${\mathbb Z}_3
\times {\mathbb
Z}_{15}$&${\mathbb Z}^2 \rtimes {\mathbb Z}_{15}$\\
   6&$1/2^2$&$2, 4, 10$&$2, 4, 6$&${\mathbb Z}_2 \times {\mathfrak
S}_5$&1&${\mathbb Z}_2 \times
{\mathbb Z}_4$&${\mathfrak S}_3 \times D_{4,5,-1}$\\
   6&$1/2^2$&$2, 7^2$&$3^2, 4$&PSL(2,7)&2&${\mathbb Z}_{21}$ &
${\mathbb Z}_7 \times {\mathfrak
A}_4$\\
   6&$1/2^2$&$2, 5^2$&$3^2, 4$&${\mathfrak A}_6$&2&${\mathbb
Z}_{15}$&${\mathbb Z}_5 \times
{\mathfrak A}_4$\\
\hline\hline
   5&$1/3, 2/3$&$2, 4, 6$&$2^4, 3$&${\mathbb Z}_2 \times {\mathfrak
S}_4$&1&${\mathbb Z}_2^2 \times
{\mathbb Z}_4$&$ 1\rightarrow {\mathbb Z}^2 \rightarrow \pi_1
\rightarrow  D_{2,8,3} \rightarrow 1$\\
   5&$1/3, 2/3$&$2^4, 3$&$3, 4^2$&${\mathfrak S}_4$ & 1 &   ${\mathbb
Z}_2^2 \times {\mathbb
Z}_8$&$ 1 \rightarrow {\mathbb Z}^2 \rightarrow \pi_1 \rightarrow
{\mathbb Z}_8 \rightarrow 1$\\
   5&$1/3, 2/3$&$4^2, 6$&$2^3, 3$&${\mathbb Z}_2 \times {\mathfrak
S}_4$ &1& ${\mathbb Z}_2^2 \times
{\mathbb Z}_8$& $1 \rightarrow {\mathbb Z}^2 \rightarrow \pi_1
\rightarrow {\mathbb Z}_8 \rightarrow
1$\\
   5&$1/3, 2/3$ & $2, 5, 6$ &$3, 4^2$&${\mathfrak S}_5$ &1&${\mathbb
Z}_8$ &              $
D_{8,5,-1}$              \\
   5&$1/3, 2/3$ & $3, 5^2$ &$2^3, 3$&${\mathfrak A}_5$ &1&${\mathbb
Z}_2 \times {\mathbb Z}_{10}$ &
${\mathbb Z}_5 \times Q_8$               \\
   5&$1/3, 2/3$&$2^3, 3$ &$3, 4^2$ &${\mathbb Z}_2^4 \rtimes {\mathfrak
S}_3$ &1& ${\mathbb Z}_2
\times {\mathbb Z}_8$ & $D_{8,4,3}$?               \\
   5&$1/3, 2/3$&$3, 5^2$ &$2^3, 3$ &${\mathfrak A}_5$ &1& ${\mathbb
Z}_2 \times {\mathbb Z}_{10}$ &
${\mathbb Z}_2 \times {\mathbb Z}_{10}$              \\
\hline\hline
   4&$1/2^4$&$2^5$&$2^5$&${\mathbb Z}_2^3$ &1& ${\mathbb Z}_2^3 \times
{\mathbb Z}_4$ &
$1\rightarrow {\mathbb Z}^4 \rightarrow \pi_1 \rightarrow {\mathbb
Z}_2^2 \rightarrow 1$       \\
   4&$1/2^4$&$2^2, 4^2$ &$2^2, 4^2$&${\mathbb Z}_2 \times {\mathbb
Z}_4$ &1& ${\mathbb Z}_2^3
\times {\mathbb Z}_4$ &   $1\rightarrow {\mathbb Z}^4 \rightarrow
\pi_1 \rightarrow {\mathbb Z}_2^2
\rightarrow 1$       \\
   4&$1/2^4$&$2^5$&$2^3, 4$&${\mathbb Z}_2 \times D_4$ &1& ${\mathbb
Z}_2^2 \times {\mathbb
Z}_4$ & $1\rightarrow {\mathbb Z}^2 \rightarrow \pi_1 \rightarrow
{\mathbb Z}_2 \times {\mathbb Z}_4
\rightarrow 1$ \\
   4&$1/2^4$&$3, 6^2$&$2^2, 3^2$&${\mathbb Z}_3 \times {\mathfrak S}_3$
&1&${\mathbb Z}_3^2$
&${\mathbb Z}^2 \rtimes {\mathbb Z}_3$           \\
   4&$1/2^4$&$3, 6^2$&$2, 4, 5$&${\mathfrak S}_5$ &1&${\mathbb Z}_3^2$
&${\mathbb Z}^2 \rtimes
{\mathbb Z}_3$          \\
   4&$1/2^4$&$2^5$&$2, 4, 6$&${{\mathbb Z}_2 \times \mathfrak S}_4$
   &1&${\mathbb Z}_2^3$ & ${\mathbb Z}^2 \rtimes {\mathbb Z}_2$ \\
   4&$1/2^4$&$2^2, 4^2$&$2, 4, 6$&${\mathbb Z}_2 \times {\mathfrak
S}_4$ &1&${\mathbb Z}_2^2 \times
{\mathbb Z}_4$ & ${\mathbb Z}^2 \rtimes {\mathbb Z}_4$\\
   4&$1/2^4$&$2^5$&$3, 4^2$&${\mathfrak S}_4$ &1& ${\mathbb Z}_2^2
\times {\mathbb Z}_4$ &
${\mathbb Z}^2 \rtimes {\mathbb Z}_4$\\
   4&$1/2^4$&$2^3, 4$&$2^3, 4$&${\mathbb Z}_2^4 \rtimes {\mathbb Z}_2$
&1& ${\mathbb Z}_4^2$ &
$G(32, 2)$\\
   4&$1/2^4$&$2, 5^2$&$2^2, 3^2$&${\mathfrak A}_5$ &1& ${\mathbb
Z}_{15}$ & ${\mathbb
Z}_{15}$                  \\
   4&$1/2^4$&$2^2, 3^2$&$2^2, 3^2$& ${\mathbb Z}_3^2 \rtimes Z_2$ &1&${\mathbb
Z}_3^3$ &
${\mathbb Z}_3^3$                  \\
   4&$2/5^2$&$2^3, 5$&$3^2, 5$&${\mathfrak A}_5$ &1&    ${\mathbb Z}_2
\times {\mathbb Z}_6$  &
${\mathbb Z}_2 \times {\mathbb Z}_6$  \\
   4&$2/5^2$&$2, 4, 5$&$4^2, 5$& ${\mathbb Z}_2^4 \rtimes D_5$ &3&
${\mathbb Z}_8$  &
${\mathbb Z}_8$?                   \\
   4&$2/5^2$&$2, 4, 5$&$3^2, 5$& ${\mathfrak A}_6$ &1&
${\mathbb Z}_6$  &  ${\mathbb
Z}_6$                   \\
\hline
\end{tabular}
\end{table}
\begin{table}
\caption{Minimal product-quotient surfaces of general type with $p_g=0$, $K^2\leq 3$}
\label{K2<3}
\renewcommand{\arraystretch}{1,3}
\tiny \begin{tabular}{|c|c|c|c|c|c|c|c|}
\hline
$K^2_S$&Sing X&$t_1$&$t_2$&$G$&N&$H_1(S,{\mathbb Z})$&$\pi_1(S)$\\
\hline\hline
   3&$1/5, 4/5$ &$2^3, 5$&$3^2, 5$& ${\mathfrak A}_5$ &1&    ${\mathbb
Z}_2 \times {\mathbb Z}_6$
&              $ {\mathbb Z}_2 \times {\mathbb Z}_6$               \\
   3&$1/5, 4/5$ &$2, 4, 5$&$4^2, 5$& ${\mathbb Z}_2^4 \rtimes D_5$ &3&
${\mathbb Z}_8$
&                $ {\mathbb Z}_8 $?                 \\
   3&$1/3, 1/2^2, 2/3$ & $2^2, 3, 4$ &$2, 4, 6$&    ${\mathbb Z}_2
   \times {\mathfrak S}_4$ &1&    ${\mathbb Z}_2 \times {\mathbb Z}_4$
&              $ {\mathbb Z}_2 \times
{\mathbb Z}_4$               \\
   3&$1/5, 4/5$&$2, 4, 5$&$3^2, 5$&    ${\mathfrak A}_6$ &1&
${\mathbb Z}_6$  &
${\mathbb Z}_6$                   \\
\hline\hline
   2&$1/3^2, 2/3^2$&$2, 6^2$ &$2^2, 3^2$& ${\mathbb Z}_2 \times
{\mathfrak A}_4$ &1& ${\mathbb
Z}_2^2$ & $Q_8$                   \\
   2&$1/2^6$&$4^3$ &$4^3$ &${\mathbb Z}_4^2$ &1&${\mathbb Z}_2^3$
&${\mathbb Z}_2^3$
\\
   2&$1/2^6$&$2^3, 4$ &$2^3, 4$ &${\mathbb Z}_2 \times D_4$ &1&
${\mathbb Z}_2 \times {\mathbb
Z}_4$ & $  {\mathbb Z}_2 \times {\mathbb Z}_4$                \\
   2&$1/3^2, 2/3^2$&$2^2, 3^2$&$3, 4^2$&${\mathfrak S}_4$ &1& ${\mathbb
Z}_8$ & ${\mathbb
Z}_8$                   \\
   2&$1/3^2, 2/3^2$&$3^2, 5$ &$3^2, 5$ &${\mathbb Z}_5^2 \rtimes
{\mathbb Z}_3$ &2&         $
{\mathbb Z}_5$ &   $              {\mathbb Z}_5$?                 \\
   2&$1/2^6$&$2, 5^2$&$2^3, 3$&${\mathfrak A}_5$ &1&${\mathbb Z}_5$
&${\mathbb Z}_5$
\\
   2&$1/2^6$&$2^3, 4$&$2, 4, 6$&${\mathbb Z}_2 \times {\mathfrak
S}_4$&1&$         {\mathbb Z}_2^2$
&$                {\mathbb Z}_2^2$                  \\
   2&$1/3^2, 2/3^2$&$3^2, 5$ &$2^3, 3$ &${\mathfrak A}_5$ &1&${\mathbb
Z}_2^2$ & ${\mathbb
Z}_2^2$                  \\
   2&$1/2^6$&$2, 3, 7$ &$4^3$ & PSL(2,7) &2& ${\mathbb Z}_2^2$
&${\mathbb Z}_2^2$                  \\
   2&$1/2^6$&$2, 6^2$&$2^3, 3$&${\mathfrak S}_3 \times {\mathfrak S}_3$
&1&$           {\mathbb Z}_3$
&$                 {\mathbb Z}_3 $                  \\
   2&$1/2^6$&$2, 6^2$&$2, 4, 5$&${\mathfrak S}_5$&1&${\mathbb
Z}_3$&${\mathbb Z}_3$                  \\
   2&$1/4, 1/2^2, 3/4$&$2, 4, 7$&$3^2, 4$&     PSL(2,7) &2& ${\mathbb
Z}_3$ &   $              {\mathbb Z}_3
$                  \\
   2&$1/4, 1/2^2, 3/4$&$2, 4, 5$&$3^2, 4$&     $     {\mathfrak
A}_6$&2&  $ {\mathbb Z}_3$ &      $
{\mathbb Z}_3$                   \\
   2&$1/4, 1/2^2, 3/4$&$2, 4, 5$&$3, 4, 6$&   ${\mathfrak S}_5$ &2&
$ {\mathbb Z}_3$ &
${\mathbb Z}_3 $                  \\
\hline\hline
   1&$1/3, 1/2^4, 2/3$&$2^3, 3$ &$3, 4^2$&${\mathfrak S}_4$
&1&${\mathbb Z}_4$ & ${\mathbb
Z}_4$                   \\
   1&$1/3, 1/2^4, 2/3$&$2, 3, 7$&$3, 4^2$&     PSL(2,7) &1& ${\mathbb
Z}_2$ & ${\mathbb
Z}_2$                   \\
   1&$1/3, 1/2^4, 2/3$&$2, 4, 6$&$2^3, 3$&${\mathbb Z}_2 \times
{\mathfrak S}_4$ &1&   ${\mathbb
Z}_2$ & ${\mathbb Z}_2$ \\
\hline
\end{tabular}
\end{table}

In the sequel we shall give some consequences of the above theorem:

Comparing tables \ref{K2>4} and \ref{K2<3} with the constructions existing in the literature, as listed in table 1 of 
\cite{surveypg=0}, we note

\begin{cor}
Minimal surfaces of general type with $p_g = q = 0$ and with  $3 \leq K^2 \leq 6$ realize at 
least 45 topological types.
\end{cor}

Note that before proving the results summarized in theorem \ref{classiso}
only $12$ topological types of surfaces of general type with $p_g = q= 0$ and with $3 \leq K^2 \leq 6$ 
were known. In 2010 Cartwright and Steger (cf. \cite{cs}) constructed 11 surfaces with $K_S^2= 3$ and new mutually
different fundamental groups, see  \cite{surveypg=0}, especially table 1,  for a more
precise account on what was previously known in the literature.

In the present paper we construct 13 surfaces with new topological types (two of them were 
independently found by Cartwright and Steger).

The biggest impact on the "zoo" of surfaces of general type with $p_g=0$ of our work is the case $K_S^2 =5$:
here we raise the number of known different topological types from one to seven.

Surfaces with $p_g = 0$ are also very interesting in view of Bloch's conjecture (\cite{bloch}),
predicting that for surfaces with $p_g  =  0$
the group of zero cycles modulo rational equivalence
is  isomorphic to $\ZZ$.

Using Kimura's results (\cite{kimura}, see also \cite{gp}), the present results,
and those of the previous papers \cite{bacat}, \cite{pg=q=0}, \cite{4names}, we get the following:
\begin{cor}
All the families in theorem \ref{classiso} fulfill Bloch's conjecture, i.e., there are 77 families  of 
surfaces  of general type with
$p_g = 0$ for which Bloch's conjecture holds.
\end{cor}

\medskip\noindent
    Let us briefly illustrate the strategy of proof for the above theorem and 
point out the difficulties arising in our more general situation.

Our goal is to find  all product-quotient surfaces $S$ of general type
with $p_g = 0$.
\begin{rem}\label{general}
1) Let $S$ be a surface of general type. Then 
$p_g(S) \geq q(S) := h^1(S, \hol_S)$. In particular, $p_g = 0$ implies $q=0$. If $S$ is minimal, then $K_S^2 > 0$.

2) Let $S$ be a product-quotient surface with quotient model $X= (C_1 \times C_2)/G$. 
If $q(S) = 0$, then  $C_i/G \cong \PP^1$. If $S$ is of general type, then $g(C_i) \geq 2$.
\end{rem}

By the above, we only need to recall the definition of a special case of an orbifold surface group: a polygonal group, 
(cf. \cite{4names} for the general situation).

\begin{defin}\label{polgr}
A {\em polygonal group} of signature $(m_1, \dots m_r)$ is the  group  presented as follows:

\begin{equation*}
\mathbb T (m_1, \ldots ,m_r) :=
\langle c_1,\ldots, c_r |
    c_1^{m_1},\ldots ,c_r^{m_r}, c_1\cdot
\ldots \cdot c_r \rangle.
\end{equation*}

\end{defin}

Let $p, p_1, \dots , p_r \in \PP^1$ be $r+1$ different points and for each $1 \leq i \leq r$ choose 
a simple
geometric loop
$\gamma_i$  in $\pi_1(\PP^1 \setminus \{p_1, \dots , p_r \}, p)$ 
around $p_i$, such that $\gamma_1 \cdot \ldots \cdot \gamma_r = 1$.

Then $\mathbb T (m_1, \ldots ,m_r)$ is the factor
group
of $\pi_1(\PP^1 \setminus \{p_1, \dots , p_r \}, p)$
 by the subgroup normally generated by $\ga_1^{m_1}, \dots , \ga_r^{m_r}$.

Hence, by Riemann's existence theorem, any curve $C$ together with an action of a finite group $G$ on it
such that $C/G\cong \PP^1$ is determined (modulo automorphisms) 
by the following data:

1) the branch point set $\{p_1, \dots p_r \} \subset \PP^1$;

2) the kernel of the monodromy homomorphism $\pi_1 ( \PP^1 \setminus \{p_1,
\dots p_r \},p) \rightarrow G$ which, once chosen loops $\gamma_i$ as above, factors through 
$\BT (m_1, \ldots ,m_r)$, where $m_i$ is the branching index of $p_i$; therefore 
giving the monodromy homomorphism  is equivalent to give

2') an {\em appropriate orbifold} homomorphism
$$
\varphi \colon \BT (m_1, \ldots ,m_r) \rightarrow G,
$$
i.e., a  surjective homomorphism  such that

a) $\varphi (c_i) $ is an element of order exactly $m_i$ and

b) the {\em Hurwitz' formula} for the genus $g$ of $C$
holds:
$$ 2g - 2 = |G|\left(-2 + \sum_{i=1}^r \left(1 -
\frac{1}{m_i}\right)\right).
$$

\noindent
Therefore a product-quotient surface $S$ of general type with $p_g = 0$
determines the following data 
\begin{itemize}
\item a finite group $G$;
\item two sets of (branch) points in $\PP^1$;
\item two polygonal groups $\BT_1$ and $\BT_2$;
\item (once chosen appropriate loops as above) two appropriate orbifold homomorphisms  
$\varphi_i \colon
\BT_i
\rightarrow G$.
\end{itemize}

Vice versa, the data above determines the product-quotient surface. 

The aim is to produce a Magma code which finds all possible $(G,\BT_i, \varphi_i)$ yielding 
surfaces of general type with $p_g=0$.

First of all we use the combinatorial restriction imposed by the  assumption
$p_g = 0$, and the condition that the quotient model of a product-quotient
surface can only have cyclic quotient singularities.

This allows, for each value of $K^2 : = K^2_S$, to restrict to a finite
number of {\em baskets} of singularities (i.e., the combinatorial data given by
the singular locus of $X$) and for each possible basket of singularities to a
finite list of possible signatures $t_1$, $t_2$ of the respective polygonal
groups.

Using prop. \ref{finite}, a
MAGMA (\cite{magmaref}) script provides a finite list of possible signatures $t_1$, $t_2$ of the respective polygonal groups.
The order of $G$ is now determined by $t_1$, $t_2$ and by
 $K^2$: it follows that there are only finitely many groups to
 consider.

A second MAGMA script computes, for each
$K^2$ and each possible basket $\mathcal{B}$, all possible triples $(t_1,t_2,G)$, where $G$ is a quotient of both
polygonal groups (of respective signatures $t_1,t_2$) and has the right
 order.
Note that our code skips a few pairs of signatures
giving rise to groups of large order, either not covered by the MAGMA
SmallGroup database, or causing extreme computational complexity.
These cases left out by our program are then excluded via a case by
 case argument.

% 1024, or bigger than 2000, which are not
%covered by the MAGMA SmallGroup database of finite groups.
%The code skips also the case $|G|=1152$, since there are more than
%$10^6$ groups of this
%order and this causes extreme computational complexity.  All the
%cases left out by our
%program are then excluded via case by case computer calculation.

For each of the triples $(t_1,t_2,G)$ in the output,
there are several pairs of surjections $(\varphi_1, \varphi_2)$, each giving a product-quotient surface.

Recall that  the triple $(t_1,t_2,G)$ depends on a previously  fixed  basket $\mathcal{B}$. 
The product-quotient surface is a surface of general type with $p_g=0$ and $K_S^2 = K^2$
if and only if  the singularities  are  as prescribed.

A third MAGMA script produces the final list of
surfaces, discarding the ones whose singular locus is not correct. 

Observe that changing the choice of the loops $\gamma_i$ (independently on both factors) 
and  changing  the $G$-action simultaneously on both Riemann surfaces by an automorphism of $G$, 
changes $(\varphi_1, \varphi_2)$, but does not change the isomorphism type of the 
resulting surface.
Therefore the script returns
only one representative for each equivalence class.

A last script calculates, using a result by Armstrong
(\cite{armstrong1}, \cite{armstrong2}), the
fundamental groups.

In the case of infinite
fundamental groups  the structure theorem proven in \cite{4names} turns out to be extremely helpful to give an explicit description of these groups
(since in general a presentation of a group does not say much about it).

Our code produces 73 families of product-quotient surfaces with $K_S^2>0$. 
While in the previous articles
we have been done, here we do not know whether the minimal resolution of singularities $S$ of 
the quotient 
model is in fact minimal. 
We have to develop methods in order to decide whether the product-quotient surface $S$ is minimal, and in case
it is not, to find the rational $(-1)$-curves on it. The construction of 
our surfaces is purely algebraic, the way from algebra to geometry is given by the Riemann existence 
theorem, which is not constructive. Therefore it is not straightforward how to get hold on delicate
geometrical features of $S$.

We develop a criterion for the minimality of $S$ arguing on the combinatorics of the basket of singularities.
In all cases except one this criterion works and the minimality of $S$ follows. 

The remaining case turns out to be non minimal. We construct two very special singular $G$-invariant 
correspondences
between $C_1$ and $C_2$ such that the respective strict transforms on $S$ of their images in $X$ are 
rational  $(-1)$-curves. We prove that the surface obtained contracting these two  $(-1)$-curves
is minimal applying the previous criterion.

The paper is organized as follows:

In section 1 we discuss finite group actions on a product of compact Riemann
surfaces of respective genera at least two, developping all the theory necessary 
to implement the algorithm.

%Moreover, we prove all the
%inequalities and numerical properties of the algebraic data associated to a
%product-quotient surface with given invariants $(\chi, K^2)$, which are then
%used to write a MAGMA code to systematically search for all minimal
%product-quotient surfaces of general type with $p_g=0$.

In the second and third section we discuss the main classification algorithm.

Section \ref{rational} deals with rational curves of selfintersection (-1) on
product-quotient surfaces. Here we give the criterion for the minimality of $S$, and show that it works for all the constructed surfaces except 
the fake Godeaux.

%In fact, all  surfaces in the output of the MAGMA code (for $K^2 \geq 1$), with the exception of exactly
%one,  turn out to be minimal surfaces of general type.

%The only product-quotient surface $S$ with $K_S^2 \geq 1$ which is non minimal
%has $K_S^2 = 1$ and fundamental group $\ZZ/6\ZZ$  (note that, by \cite{tokyo},
%if  $S$ were minimal, it would follow that $|\pi_1(S)| \leq 5$).

In section \ref{fakegodeaux} we determine the minimal model of the fake Godeaux surface.
The last section is devoted to comments about the computational complexity of the algorithms we used.

Finally, in a first appendix, we attach an expanded version of tables
\ref{K2>4}, \ref{K2<3} describing all the needed data if one wants to do explicit computations with 
one of the surfaces.

The second appendix is the MAGMA code we used.

\section{Theoretical background}

Let $C_1$, $C_2$ be two compact Riemann surfaces of respective genera
$g_1, g_2 \geq 2$. Let $G$ be a finite group acting faithfully on both curves and consider the diagonal action of $G$ on $C_1 \times C_2$. This determines a product-quotient surface $S$, the minimal resolution of the singularities of $X:=(C_1 \times C_2)/G$.

\begin{rem}\label{vanKampen} 1) Note that there are finitely many points on $C_1 \times C_2$ with non trivial stabilizer, which is automatically cyclic. Hence the quotient surface $X:= (C_1 \times C_2) / G$ has
a finite number of cyclic quotient singularities. 

Recall that every cyclic
quotient singularity is locally analytically isomorphic to the quotient of $\mathbb{C}^2$ by
      the action of a diagonal linear automorphism with eigenvalues
      $\exp(\frac{2\pi i}{n})$, $\exp(\frac{2\pi i a}{n})$ with $g.c.d(a,n) = 1$; this is called a {\em singularity of type $\frac{1}{n}(1,a)$}.

2) We denote by $K_X$ the canonical (Weil) divisor on the normal surface $X$
corresponding to $i_* ( \Omega^2_{X^0})$, $ i\colon X^0 \ra X$ being the
inclusion of the
smooth locus of $X$. According to Mumford we have an intersection
product with values
in $\QQ$ for Weil divisors on a normal surface, and in particular we may consider
the selfintersection  of the canonical divisor,
\begin{equation}\label{K2}
K_X^2 =
\frac{8 (g(C_1) - 1) (g(C_2) - 1)}{|G|}
\in
      \mathbb{Q},
\end{equation}
 which is not necessarily an integer.

3) It is well known that the exceptional divisor $E$ of the minimal resolution of a cyclic quotient singularity of type  $\frac{1}{n}(1,a)$ is a {\em Hirzebruch-Jung string}, i.e., $E = \bigcup_{i=1}^l E_i$ where all $E_i$ are smooth rational curves, $E_i^2 = -b_i$, $E_i\cdot E_{i+1}
= 1$ for $i \in\{1, \ldots , l-1\}$ and $E_i\cdot E_j=0$ otherwise. The $b_i$ are given by the formula
$$
\frac{n}{a} = b_1 - \frac{1}{b_2 - \frac{1}{b_3 - \ldots}}.
$$

4) Since the minimal resolution  $\pi \colon S
\rightarrow X$ of the
singularities of $X$ replaces each singular point by a tree of smooth
rational curves, we have, by van Kampen's theorem, that
$\pi_1(X) = \pi_1(S)$.

5) Moreover, we have (in a neighbourhood of $x$)
$$ K_S = \pi^* K_X + \sum_{i=1}^l a_i E_i,
$$  where the rational numbers $a_i$ are determined by the conditions
    $$(K_S + E_j)E_j =
-2, \ \ \
(K_S -
\sum_{i=1}^l a_iE_i)E_j = 0, \ \
\forall j= 1, \dots ,l.$$
\end{rem}

The  above formulae  allow us to calculate
the self intersection
number of the canonical divisor $K_S$. In fact, we need the following
\begin{defin}
Let $X$ be a normal complex surface and suppose that the singularities of $X$ are cyclic quotient
singularities. Then we define  the {\em basket of singularities of $X$}  to be  the multiset
$$
\mathcal{B}(X) := \left\{\lambda \times \left(\frac{1}{n} (1,a)\right) : X \ {\rm has \ exactly} \ \lambda \
{\rm singularities }  \ {\rm  of \ type} \ \frac{1}{n} (1,a) \right\}.
$$
\end{defin}

I.e., $\mathcal{B}(X) = \{2 \times \frac 13 (1,1), \frac 14 (1,3) \}$ means that the singular locus of $X$ consists of two $\frac 13 (1,1)$-points and one $\frac 14 (1,3)$-point.

\begin{rem}\label{ambiguity}
Note that in the definition of ${\mathcal B}(X)$ there is some ambiguity since singular points of type
$\frac1n(1,a)$ are also of type  $\frac1n(1,a')$ where $a' = a^{-1}$ in
$(\ZZ / n \ZZ)^*$. Therefore, {\it e.g.},
$$\{2 \times \frac 15 (1,2) \}=
\{1 \times \frac 15 (1,2), 1 \times \frac 15 (1,3) \}=
\{2 \times \frac 15 (1,3) \}.$$
We consider these different representations as equal and usually do not distinguish between them.
\end{rem}

\begin{defin}
Let $x$ be a singularity of type $\frac 1n (1,a)$ with
$\gcd(n,a) = 1$ and let $1 \leq a' \leq n-1$ such that $a' = a^{-1}$ in
$(\ZZ / n \ZZ)^*$. Moreover, write $\frac na$ as a continued fraction:
$$
\frac{n}{a} = b_1 - \frac{1}{b_2 - \frac{1}{b_3 - \ldots}} =:[b_1, \ldots, b_l].
$$
Then we define the following correction terms:

\begin{itemize}
 \item[i)] $k_x := k(\frac 1n (1,a)):= -2 + \frac{2+a+a'}{n} + \sum(b_i-2) \geq 0$;
\item[ii)] $e_x := e(\frac 1n (1,a)):= l + 1 - \frac 1n \geq 0$;
\item[iii)] $B_x := 2e_x + k_x$.
\end{itemize}
\end{defin}

 Let ${\mathcal B}$ be the basket of singularities of $X$ (recall that $X$ is normal and has only cyclic quotient singularities).
Then we use the following notation
$$k({\mathcal B}):=\sum_{x \in {\mathcal B}} k_x, \ \ \ 
e({\mathcal B}):=\sum_{x \in {\mathcal B}} e_x, \ \ \ B({\mathcal B}):=\sum_{x \in {\mathcal B}} B_x.
$$

\begin{prop}[\cite{4names}, prop. 2.6, and \cite{polmi}, cor. 3.6]\label{k2e}
Let $S \ra X:=( C_1 \times C_2 )/G$ be the minimal resolution of singularities of $X$. Then we have the following two formulae for the self intersection of the canonical divisor of $S$ and the topological Euler characteristic of $S$:
\begin{equation*}
 K_S^2 = \frac{8 (g_1 -1)(g_2-1)}{|G|} - k({\mathcal B});
\end{equation*}
\begin{equation*}
e(S) = \frac{4 (g_1 -1)(g_2-1)}{|G|} + e({\mathcal B}).
\end{equation*}
\end{prop}

A direct consequence of the above is the following:
\begin{cor}\label{k2basket}
Let $S \ra X:=( C_1 \times C_2 )/G$ be the minimal resolution of singularities of $X$. Then
$$
K_S^2 = 8 \chi(S) - \frac 13 B({\mathcal B}).
$$
\end{cor}

\begin{proof}
 By prop. \ref{k2e} we have 
$$
e(S) = \frac{K_S^2 + B(\mathcal{B})}{2}.
$$
By Noether's  formula we obtain
$$
12\chi(S)=K^2_S+e(S) = \frac{3K_S^2 + B(\mathcal{B})}{2}
$$
\end{proof}

We shall now list some properties of the basket of singularities of the quotient model $X = ( C_1 \times C_2 )/G$ of a product-quotient surface.

\begin{lemma}
 Let $X=( C_1 \times C_2 )/G$ be as above. There exists a representation of the basket (cf. remark \ref{ambiguity})
$$
\mathcal{B}(X)=\left\{\lambda_1 \times \frac{1}{n_1} (1,a_1), \ldots, \lambda_R \times \frac{1}{n_R} (1,a_R) \right\}
$$
 such that
$$
\sum \lambda_i \frac{a_i}{n_i} \in \ZZ.
$$
\end{lemma}

\begin{proof}
Consider the fibration $X \rightarrow C_1/G$, and let $F_1,\ldots, F_r$ be the singular fibres
taken with the reduced structure.
Let $\tilde{F}_i$ be the strict transform of $F_i$ on $S$.

Then, by \cite[Proposition 2.8]{PolizziNumProp}, for a suitable representation of the basket
$$
\sum \lambda_i \frac{a_i}{n_i}=-\sum \tilde{F}_i^2\in \ZZ.
$$
\end{proof}

\begin{defin}\label{possible}
A multiset
$$
\mathcal{B}:=\left\{\lambda_1 \times \frac{1}{n_1} (1,a_1), \ldots, \lambda_R \times \frac{1}{n_R} (1,a_R)\right\}
$$
is called a {\em possible basket of singularities for $(K^2, \chi)$} if and only if it satisfies the following conditions:
\begin{itemize}
\item there is a representation of ${\mathcal B}$, say $$
\mathcal{B}:=\left\{\lambda'_1 \times \frac{1}{n'_1} (1,a'_1), \ldots, \lambda'_{R'} \times \frac{1}{n'_{R'}} (1,a'_{R'}) \right\}
$$
such that $\sum \lambda'_i \frac{a'_i}{n'_i} \in \ZZ$,
\item $B(\mathcal{B})=3(8 \chi(S)-K^2)$.
\end{itemize}

\end{defin}

It is now obvious that the basket of the quotient model $X$ of a product-quotient surface $S$ is a possible basket of singularities for $(K_S^2, \chi(\hol_S))$.

\subsection{Finiteness of the classification problem}

The next lemma shows that, for every pair $(K^2, \chi) \in \ZZ \times \ZZ$, there are only finitely many  possible baskets of singularities for $(K^2, \chi)$.

\begin{lemma}\label{poss}
 Let $C \in \QQ$ be fixed. Then there are finitely many baskets $\mathcal{B}$ such that
$$
B({\mathcal B}) = C.
$$

More precisely, we have:
\begin{itemize}
 \item[i)] $|\mathcal{B}| \leq \frac{C}{3}$,
 \item[ii)] if $\lambda \times \frac{1}{n}(1,a) \in \mathcal{B}$ and $\frac{n}{a}=[b_1,\dots,b_l]$, then
$\lambda \sum b_i \leq C$.
%\item[ii)] $n_{max} \leq 3(8 \chi(S) - K^2)$, where $n_{max} :=
%\max\{ n \in \NN : \exists \lambda>0 \ {\rm s.th.} \ \lambda \times \frac{1}{n}(1,a) \in \mathcal{B}\}$.
\end{itemize}

\end{lemma}

\begin{proof}
 Observe first that  $B(\frac{1}{n}(1,a))= \frac{a+a'}{n} + \sum b_i\geq 3$. In particular,
$$
C =B(\mathcal{B}) \geq 3|\mathcal{B}|,
$$
which shows $(i)$. $(ii)$ is obvious.

%Let $x \in \Sing(X)$ be a point of type $\frac{1}{n}(1,a)$, then obviously: $\sum b_i \geq n$, whence $n \leq B(\frac{1}{n}(1,a))$. Therefore $n_{max} \leq \sum_{k \times \frac{1}{n}(1,a) \in \mathcal{B}} k \cdot B(\frac{1}{n}(1,a))  = 3(8 \chi(S)-K^2)$.
\end{proof}

\begin{rem}
Note that, by \cite{serrano}, if $S\rightarrow X=C_1 \times C_2/G$ is a product-quotient surface,
then $q(S)=g(C_1/G)+g(C_2/G)$. Therefore 
$q(S)=0 \Leftrightarrow g(C_1/G)=g(C_2/G)=0$.
This implies that a product-quotient surface $S$ of general type with quotient model $X = (C_1 \times C_2)/G$  has $p_g(S)=0$ if and only if 
\begin{itemize}
 \item $\chi(\hol_S) = 1$ and 
\item $C_1/G \cong C_2/G \cong \PP^1$.
\end{itemize}

\end{rem}
From now on we shall restrict ourselves to product-quotient surfaces $S$ of general type with $p_g(S) =0$. Let $\lambda_i \colon C_i \rightarrow
\PP^1$, $i = 1, 2$ be the two Galois covers associated to it.

Recall that, by Riemann's existence theorem (cf. the introduction for more details), an action  of a finite group $G$ on a
compact Riemann
surface $C$ of genus $g$ such that $C/G \cong \PP^1$ is given by  an appropriate  orbifold homomorphism
$$\varphi \colon
\BT(m_1,\ldots, m_r)
\rightarrow G$$  such that the Riemann-Hurwitz relation holds:
$$ 2g - 2 = |G|\left(-2 + \sum_{i=1}^r \left(1 -
\frac{1}{m_i}\right)\right).
$$
Then $\lambda_i \colon C_i \rightarrow
\PP^1$, $i = 1, 2$ induce two appropriate orbifold homomorphisms
$$
\varphi_1 \colon \BT(m_1,\ldots, m_r) \rightarrow G,
$$
$$
\varphi_2 \colon \BT(n_1,\ldots, n_s) \rightarrow G.
$$
Here $\lambda_1$ is branched in $r$ points $p_1, \ldots , p_r \in \PP^1$ with branching
indices $m_1, \ldots ,
m_r$, and  $\lambda_2$ is branched in $s$ points $p_1', \ldots , p_s' \in \PP^1$ with branching
indices $n_1, \ldots ,
n_s$.

We need the following

\begin{defin}
Fix an $r$-tuple of natural numbers $t:=(m_1, \ldots , m_r)$ and a basket of singularities $\mathcal{B}$. Then we associate to these the following numbers:
$$
\Theta(t) := -2 + \sum_{i=1}^r (1 - \frac {1}{m_i});
$$
$$
\alpha(t, \mathcal{B}) :=
%\frac{K^2 + k(\mathcal{B})}{4 \Theta(t)}=
\frac{12 + k(\mathcal{B})- e(\mathcal{B})}{6 \Theta(t)}.
$$
\end{defin}

Moreover, we recall the following
\begin{defin}
The minimal positive integer $I_x$ such that $I_xK_X$ is Cartier in
$x$ is called the {\em index of the singularity $x$}.

The {\em index of $X$} is the minimal positive integer $I$ such that $IK_X$ is Cartier. In particular,
$I=\lcm_{x \in \Sing X} I_x$.
\end{defin}

It is well known (cf. e.g. \cite{matsuki}, theorem 4-6-20) that the index of a cyclic quotient singularity $\frac{1}{n}(1,a)$  is
$$I_x=\frac{n}{\gcd(n,a+1)}.
$$

\medskip\noindent
 By lemma \ref{poss}, fixed $K^2 \in \ZZ$, there are finitely many possible baskets of singularities  for $(K^2, \chi(\hol_S) = 1)$. 

\medskip\noindent
We shall bound now, for fixed $K^2$ and $\mathcal{B}$, the possibilities for:
\begin{itemize}
 \item $|G|$,
\item $t_1:=(m_1, \ldots, m_r)$,
\item $t_2:=(n_1, \ldots, n_s)$,
\end{itemize}
of a product-quotient surface $S$ with $K_S^2 = K^2$ and basket of singularities of the quotient model $X$ equal to $\mathcal{B}$.

\begin{prop}\label{finite}
 Fix $K^2 \in \ZZ$, and fix a possible basket of singularities $\mathcal{B}$ for $(K^2,1)$. Let $S$ be a product-quotient surface $S$ of general type such that
\begin{itemize}
 \item[i)] $p_g(S) =0$,
\item[ii)] $K_S^2 = K^2$,
\item[iii)] the basket of singularities of the quotient model $X=(C_1\times C_2)/G$ of $S$ equals $\mathcal{B}$.
\end{itemize}
Then:
\begin{itemize}
\item[a)] $g(C_1)=\alpha(t_{2},{\mathcal B})+1$, $g(C_2)=\alpha(t_{1},{\mathcal B})+1$;
\item[b)] $|G| = \frac{8 \alpha(t_1, \mathcal{B})\alpha(t_2, \mathcal{B})}{K^2+k({\mathcal B})}$;
 \item[c)] $r,s \leq \frac{K^2+k(\mathcal{B})}{2} +4$;
\item[d)] $m_i$ divides $2\alpha(t_1,{\mathcal B})I$, $n_j$ divides $2\alpha(t_2,{\mathcal B})I$;
\item[e)] there are at most $|{\mathcal B}|/2$ indices $i$ such that $m_i$ does not divide
$\alpha(t_1,{\mathcal B})$, and similarly for the $n_j$;
\item[f)] $m_i \leq \frac{1+I \frac{K^2+k(\mathcal{B})}{2}}{f(t_1)}$,
$n_i \leq \frac{1+I \frac{K^2+k(\mathcal{B})}{2}}{f(t_2)}$,
where $I$ is the index of $X$,
and $f(t_1):= \max (\frac{1}{6}, \frac{r-3}{2})$, $f(t_2):= \max (\frac{1}{6}, \frac{s-3}{2})$;
\item[g)] except for at most $|{\mathcal B}|/2$ indices $i$, the sharper inequality
$m_i \leq \frac{1+ \frac{K^2+k(\mathcal{B})}{4}}{f(t_1)}$ holds, and similarly for the $n_j$.
\end{itemize}
\end{prop}

\begin{rem}
 Note that prop. \ref{finite}, b) shows that $t_1$, $t_2$ determine the order of $G$. c), f) imply that there are only finitely many possibilities for the types $t_1$, $t_2$.
Parts d), e) and g) are strictly necessary to obtain an efficient algorithm.
\end{rem}

\begin{proof}
a) Observe that by corollary \ref{k2basket}, since $\chi(\hol_S)=1$, we have
$$\Theta(t_1) \alpha(t_1,\mathcal{B}) = \frac{12 +k(\mathcal{B})-e(\mathcal{B})}{6}=
\frac{24 -B(\mathcal{B})+3k(\mathcal{B})}{12}=
\frac{K^2+k(\mathcal{B})}{4}
$$
and then by prop. \ref{k2e} and Hurwitz' formula
$$
\alpha(t_1, \mathcal{B})=
\frac{K^2 +k(\mathcal{B})}{4 \Theta(t_1)}
= \frac{8(g(C_1)-1)(g(C_2)-1)}{4|G| (-2 + \sum_{i=1}^r (1 - \frac {1}{m_i}))}
= \frac{8(g(C_1)-1)(g(C_2)-1)}{4(2g(C_1)-2)}.
$$

b) $$|G|= \frac{8 (g(C_1)-1)(g(C_2)-1)}{K^2+k({\mathcal B})}
= \frac{8 \alpha(t_2, \mathcal{B})\alpha(t_1, \mathcal{B})}{K^2+k({\mathcal B})}.$$

c) Note that $r \leq 2\sum_{i=1}^r (1 - \frac{1}{m_i})= 2 \Theta(t_1) +4$.
On the other hand, since $g(C_j) \geq 2$, we have $1 \leq \alpha(t_i, \mathcal{B}) =
\frac{K^2+k(\mathcal{B})}{4 \Theta(t_i)}$.
This implies that $(0 <) \ \Theta(t_i) \leq \frac{K^2+k(\mathcal{B})}{4}$.

d) Each $m_i$ is the branching index of a branch point $p_i$ of  $\lambda_1 \colon C_1 \rightarrow C_1/G\cong \PP^1$.
Let $F_i$ be the fibre over $p_i$ of the map $X \rightarrow C_1/G$. Then $F_i=m_iW_i$ for some irreducible Weil divisor
$W_i$.
$$
2\alpha(t_1,{\mathcal B})=2g(C_2)-2=K_XF_i=m_iK_XW_i.
$$
Therefore
$$
\frac{2\alpha(t_1,{\mathcal B})I}{m_i}=(IK_X)W_i \in \ZZ.
$$
e) By \cite{serrano}, if $F_i$ contains a singular point of $X$, then it contains at least $2$ singular points.
Therefore there are at most $|{\mathcal B}|/2$ indices $i$ ($1 \leq i \leq r$) such that $F_i \cap \Sing X \neq \emptyset$.

For all other indices $j$ we have $F_j \cap \Sing X=\emptyset$. Then $W_j$ is Cartier and $K_X$ is Cartier in a neighbourhood of $W_j$.
In particular, $\frac{\alpha(t_1,{\mathcal B})I}{m_j}=\frac{K_XW_j}2 \in \ZZ$.

f) Note that $\Theta(t_1)+\frac{1}{m_i} \geq \frac{r-3}2$.
Moreover, $\Theta(t_1)>0$ implies that $r\geq 3$. 
Obviously, if $r=3$, since $\Theta(2,2,m)=-\frac1m<0$, then $\Theta(t_1)+\frac{1}{m_i}\geq \frac16$.
Therefore $\Theta(t_1)+\frac{1}{m_i}\geq f(t_1)$, whence $m_i\leq \frac{1+\Theta(t_1)m_i}{f(t_1)}$.

By d) $m_i\leq 2\alpha(t_1,{\mathcal B})I=\frac{K^2+k({\mathcal B})}{2\Theta(t_1)}I$.
This implies
$$m_i \leq \frac{1+\Theta(t_1)m_i}{f(t_1)}\leq
\frac{1+\Theta(t_1)\frac{K^2+k({\mathcal B})}{2\Theta(t_1)}I}{f(t_1)} =
\frac{1+\frac{K^2+k({\mathcal B})}{2}I}{f(t_1)}.$$

g) This is proved by the same argument as in f), using e) instead of d).
\end{proof}

\subsection{How to read the basket $\mathcal{B}$ from the group theoretical data}
Our next goal is to describe explicitly how the two appropriate orbifold homorphisms
$$
\varphi_1 \colon \BT(m_1,\ldots, m_r) \rightarrow G,
$$
$$
\varphi_2 \colon \BT(n_1,\ldots, n_s) \rightarrow G.
$$
determine the singularities of the quotient model $X$. 

We denote the images of the standard generators (the $c_i$ in definition \ref{polgr}) of $\BT(m_1,\ldots, m_r)$ (resp. of 
$\BT(n_1,\ldots, n_s)$) by $(g_1, \ldots, g_r)$ (resp. by $(h_1, \ldots, h_s)$).

Moreover we set $H_i:=\langle g_i \rangle$ and $H'_j:=\langle h_j \rangle$.

We have now the following commutative diagram:
\begin{equation}\label{diagram1}
\xymatrix{
&C_1 \times C_2 \ar^{p_2}[dr]\ar_{p_1}[dl]\ar_{\lambda_{12}}[dd]&\\
C_1\ar_{\lambda_1}[dd]&&C_2\ar^{\lambda_2}[dd]\\
&X = (C_1 \times C_2)/G\ar_{f_1}[ld]\ar^{f_2}[dr]\ar^{\lambda}[dd]&\\
C_1/G \cong \PP^1&&C_2/G \cong \PP^1\\
&C_1/G \times C_2/G \cong \PP^1 \times \PP^1 \ar_{}[ul]\ar_{}[ur]&\\
}
\end{equation}

Note that the singular points of $X$ are the points $Q = \lambda_{12}(q,q')$ such that the stabilizer
$$
\Stab(q,q'):= \Stab(q) \cap \Stab(q') \neq \{1\}.
$$

In particular, if $Q \in \Sing(X)$ then $\lambda(Q) = (p_i, p'_j)$, where $p_i$ (resp. $p'_j$) is a critical value of $\lambda_1$ (resp. $\lambda_2$).

We first prove the following

\begin{prop}\label{bijective}
Let $i \in \{1, \ldots , r\}$, $j \in \{1, \ldots s \}$. Then
\begin{enumerate}
\item there is a $G-$equivariant bijective map $(\lambda \circ \lambda_{12})^{-1} (p_i, p'_j) \ra G/H_i \times G/H'_j$, where the $G-$action on the target is given by left multiplication (simultaneously on both factors);
\item intersecting with $\{\bar{1}\} \times G/H'_j$ gives a bijection between the orbits of the above $G-$action on $G/H_i \times G/H'_j$ with the orbits of the $H_i$-action on $G/H'_j$, {\it i.e.} with $(G/H'_j)/H_i$
%$$
%\mathcal{S}(i,j):=\pi((\pi \circ \pi_{12})^{-1} (p_i, p'_j)) = \{(gh)\mod H'_j : g \in G/H'_j, h \in H_i\};$$

\end{enumerate}

\end{prop}

\begin{proof}
1) Wlog we can assume $(i,j) = (1,1)$. We fix the following notation
$$
\pi_1^{-1}(p_1) = \{q_1, \ldots, q_k\}, \ \ \pi_2^{-1}(p'_1) = \{q'_1, \ldots, q'_l\}.
$$

There is a $G-$equivariant bijection between $\{q_1, \ldots, q_k\}$ and the set of left cosets
$$
\{ a_1H_1, \ldots , a_kH_1\},
$$
mapping each $q_j$ in $\{g\in G|gq_1=q_j\}$;
 similarly there is a bijection between $\{q'_1, \ldots, q'_l\}$ and
$$
\{ a'_1 H'_1 , \ldots , a'_l H'_1 \}.
$$

This gives a $G-$equivariant bijection between $(\lambda \circ \lambda_{12})^{-1} (p_1, p'_1)$ and $G/H_1 \times G/H'_1$. 

2) We consider the (diagonal) $G$-action on $G/H_1 \times G/H'_1$ by left multiplication. Note that the $G$-orbits are in one-to-one correspondence with the points of $\lambda((\lambda \circ \lambda_{12})^{-1} (p_1, p'_1))$.

Observe that
\begin{itemize}
 \item[i)] $(hH_1 , h'H'_1)$ is in the same $G$-orbit as $(H_1, h^{-1}h'H'_1)$;
\item[ii)] $(H_1, gH'_1)$ is in the same $G$-orbit as $(H_1,g'H'_1)$ if and only if $gH'_1$ and $g'H'_1$
are in the same orbit for the action of $H_1$.
\end{itemize}

\end{proof}

\begin{rem}
 Recall that $\Sing(X) \subset  \lambda^{-1}(\{(p_i,p'_j)\})$. Observe moreover that proposition \ref{bijective} gives for each $(i,j)$ a bijection between $\lambda^{-1} (p_i, p'_j)$ and $(G/H'_j)/H_i$.
\end{rem}

We still have to determine the types of the  singularities. This is done in the following 
\begin{prop}
An element $[g]\in (G/H'_j)/H_i$ corresponds to a point $\frac 1n (1,a)$,
where $n = |H_i \cap gH'_jg^{-1}|$, and $a$ is given as follows: let $\delta_i$
be the minimal positive number such that there exists $1 \leq \gamma_j \leq o(h_j)$ with
$g_i^{\delta_i}=gh_j^{\gamma_j}g^{-1}$. Then $a = \frac{n\gamma_j}{o(h_j)}$.
\end{prop}

\begin{proof}
 Again we can assume wlog. that $(i,j) = (1,1)$. Then $[g]$ corresponds
to a (singular) point of type $\frac 1n(1,a)$ with
$n = |\Stab(q_1,gq'_1)|= |H_1 \cap g H'_1 g^{-1}|$. Recall that $H_1 =\langle g_1 \rangle$, and $H'_1 = \langle h_1\rangle$.

Let $\delta$ be the minimal positive number such that
there is $\gamma \in \NN$ (which can be choosen such that $1 \leq \gamma \leq o(h_1)$) such that $g_1^{\delta}=g h_1^{\gamma} g^{-1}$. Then $\langle g_1^{\delta} \rangle = \Stab(q_1,gq'_1)$.

Therefore $o(g_1) = n \delta$. In local analytic coordinates $(x,y)$ of $C_1 \times C_2$, $g_1^{\delta}$ acts as
$$
e^{\frac{2\pi i}{n}} = e^{\frac{2\pi i \delta}{o(g_1)}}
$$
on the variable $x$ and  as
$$
e^{\frac{2\pi i a}{n}} = e^{\frac{2\pi i \gamma}{o(h_1)}}.
$$
on the variable $y$. This shows that $a = \frac{n \gamma}{o(h_1)}$.
\end{proof}

\section{Description and implementation of the classification algorithm}\label{algorithm}
Now we  use the results of the previous section to write a MAGMA script to find all
minimal surfaces $S$ of general type with  $p_g=0$, which are product-quotient surfaces.

The full code is rather long and we attach a commented version in the appendix.
We describe here the strategy, and explain the most important scripts.

First of all, by rem. \ref{general}, cor. \ref{k2basket}, $1 \leq K_S^2 \leq 8$. The case $K_S^2 =8$ has been classified in \cite{pg=q=0}. 

Therefore
we fix a value of $K^2 \in \{1, \ldots , 7 \}$.

{\bf Step 1}: The script {\bf Baskets} lists all the {\em possible baskets of singularities} for $(K^2,1)$ as
in definition \ref{possible}. Indeed, there are only finitely many of them by lemma \ref{poss}. The input is $3(8-K^2)$,
as in lemma \ref{poss}, so to get {\it e.g.}, all baskets for $K^2_S=5$, we need to ask
{\it Baskets(9)}.

{\bf Step 2}: By proposition \ref{finite}, once we know the basket of singularities of $X$,
then there are finitely many possible signatures. {\bf ListOfTypes} computes them using the
inequalities we have proved in proposition \ref{finite}.
Here the input is $K^2$, so ListOfTypes first computes {\it Baskets$(3(8-K^2))$} and then computes
for each basket all numerically compatible signatures. The output is a list of pairs,
the first element of each pair being a basket and the second element being the list of all signatures compatible with that basket.

{\bf Step 3:} Every surface produces two signatures, one for each curve $C_i$, both compatible with the
basket of singularities of $X$; if we know the signatures and the basket, Proposition \ref{finite}, b) tells
us the order of $G$. {\bf ListGroups}, whose input is $K^2$, first computes {\it ListOfTypes($K^2$)}. Then for each pair of signatures in the output, it calculates the order of the group. Next it searches for the groups of the given order which admit appropriate orbifold homomorphism from the polygonal groups corresponding to both signatures.
For each affirmative answer it stores the triple
(basket, pair of signatures, group) in a list which is the main output.

 The script has some shortcuts.
 \begin{itemize}
 \item If one of the signatures is (2,3,7), then $G$, being a quotient of $\BT(2,3,7)$, is perfect.
MAGMA knows all perfect groups of order $\leq 50000$, and then ListGroups checks first if there are perfect
group of the right order: if not, this case can't occur.
 \item If:
\begin{itemize}
\item either the expected order of the group is  1024 or bigger than 2000, since MAGMA does not have a list of
the finite groups of this order;
\item or the order is a number as {\it e.g.}, 1728, where there are too many isomorphism classes of groups; 
\end{itemize}
then ListGroups just stores these cases in a list, secondary output of the script. We will consider these "exceptional" cases in the next subsection, showing that they do not occur.
\end{itemize}

{\bf Step 4:} {\bf ExistingSurfaces} runs on the output of {\it ListGroups($K^2$)} and throws
away all triples giving rise only to surfaces whose singularities do not correspond to the basket.

{\bf Step 5:} Each triple in the output of {\it ExistingSurfaces($K^2$)} gives many different pairs of appropriate orbifold homomorphisms. In  \cite{bacat} (def. 1.2., thm. 1.3.) there is explicitly described an equivalence relation on such pairs of appropriate orbifold homomorphisms. By \cite{bacat} thm. 1.3. and its proof if two pairs belong to the same equivalence class then the surfaces obtained by them (as described in the introduction, choosing for both the same points and the same loops) are isomorphic. More precisely, they are product-quotient surfaces with the same group $G$ and the isomorphism is induced by a $G-$equivariant isomorphism of the related products of curves.

The script {\bf FindSurfaces} produces, given a triple
(basket, pair of types, group), only one representative for each equivalence class.

{\bf Step 6:} {\bf Pi1} uses Armstrong's result (\cite{armstrong1}, \cite{armstrong2}) to compute
the fundamental group of each of the constructed surfaces.

\begin{rem}
We performed step 5 to avoid useless repetitions (note that the cardinality of some equivalence class is a few millions). Nevertheless, it is still possible that two different outputs of FindSurfaces give isomorphic surfaces. One of the reasons for running step 6 is indeed to show that this is in many cases not true, since the fundamental group distinguishes them even topologically.

We would also like to point out that, even if our families have a natural number of parameters,  we do not make any claim on the dimension of the induced subsets of the Gieseker moduli space of the surfaces of general type. 
\end{rem}

\begin{rem}
The output of {\it Pi1} is a (sometimes rather complicated) presentation of the fundamental groups of the respective surfaces.
We use the structure theorem on the fundamental group of product-quotient surfaces \cite[Theorem 4.1]{4names}
to give the (nicer) description of the fundamental groups in tables \ref{K2>4} and \ref{K2<3}.
\end{rem}

We have run {\it FindSurfaces} on each triple of the output of {\it ListGroups($K^2$)},
$K^2 \in \{1 \ldots 7 \}$. This has given all the families in tables \ref{K2>4} and \ref{K2<3}, and one more, the ``fake Godeaux surface''. 

To prove theorem \ref{classiso} it remains to show that
\begin{itemize}
 \item all cases skipped by {\it ListGroups} do not occur,
 \item all the families in tables \ref{K2>4} and \ref{K2<3} are minimal surfaces of general type,
 \item  the ``fake Godeaux surface'' has the properties in thm. \ref{classiso}, 3).
\end{itemize}

This will be accomplished in  sections \ref{exceptional}, \ref{rational} and \ref{fakegodeaux}.

product-quotient surfaces with $p_g=q=0$ and $K^2_S \geq1$,
as soon as we prove that the cases skipped by ListGroups cannot occur. This is done in the next section.

\section{The exceptional cases}\label{exceptional}

The cases skipped by {\it ListGroups} and stored in its secondary output are listed in table \ref{tocheck}.

\begin{table}[hb]
\caption{Secondary output of {\it ListGroups}}\label{tocheck}
\begin{tabular}{|c|c|c|c|c|}
\hline
$K^2$ & Basket & $t_1$ & $t_2$ & $|G|$\\
\hline
6& $1/2^2$ & 2, 3, 7 & 2, 4, 5 & 2520\\
\hline
5& 2/3, 1/3 & 2, 3, 8 &2, 4, 6 &768  \\
5& 2/3, 1/3 & 2, 3, 8 &2, 3, 7 &2688 \\
5& 2/3, 1/3 & 2, 3, 8 &2, 3, 8 &1536 \\
5& 2/3, 1/3 & 2, 3, 8 &2, 3 ,9 &1152 \\
\hline
4&1/2, $1/4^2$ & 2, 3, 7 & 2, 4, 5 & 2520 \\
4&$1/2^4$ & 2, 3, 8 & 2, 3, 8 & 1152 \\
\hline
2&$1/4^4$&2, 4, 5 &2, 3, 7 &  2520\\
2&$1/2^3$, $1/4^2$&2, 3, 8 & 2, 3, 8& 1152\\
2&$2/3^2$, $1/3^2$&2, 3, 8 & 2, 3, 8& 768 \\
\hline
1&1/4, 1/5, 11/20&2, 3, 8 &2, 3, 8 & 2016 \\
1&$2/7^2$, 1/7& 2, 3, 7&  2, 3, 7& 6048\\
1&1/4, 2/5, 3/20&2, 3, 8 & 2, 3, 8& 2016\\
1&1/4, 5/8, 1/8& 2, 3, 8& 2, 3, 8& 2016\\
\hline
\end{tabular}
\end{table}

In this section we shall show that all these cases do not occur. One of the main tools here is the script {\bf ExSphGens}, which checks, given a finite group $G$ and a signature, the existence of an appropriate orbifold homomorphism from the polygonal group of given signature to $G$.
\begin{prop}\label{237}
There is no finite quotient of $\BT (2,3,7)$ of order 2520, 2688 or 6048.
\end{prop}

\begin{proof}
A finite quotient of $\BT (2,3,7)$ is perfect.
The only perfect groups of order 2520 resp. 6048 are ${\mathfrak A}_7$ resp. $SU(3,3)$; running the
MAGMA script {\it ExSphGens} on these two groups, it turns out that
 both cannot be a quotient of  $\BT (2,3,7)$.

There are $3$ perfect groups of order 2688.
Let $G$ be one of these three groups.
Investigating their normal subgroups we find that $G$ is either an
extension of the form
$$
1
\rightarrow
({\mathbb Z}/2{\mathbb Z})^3
\rightarrow
G
\rightarrow
SU(2,7)
\rightarrow
1
$$
or of the form
$$
1
\rightarrow
{\mathbb Z}/2{\mathbb Z}
\rightarrow
G
\rightarrow
SmallGroup(1344,11686)
\rightarrow
1.
$$

Running {\it ExSphGens} on  $SU(2,7)$ and on $SmallGroup(1344,11686)$, we see that
none of them is quotient of $\mathbb{T}(2,3,7)$. Since $\mathbb{T}(3,7)=\{1\}$, this implies that  $G$ is not a quotient of   $\mathbb{T}(2,3,7)$.
\end{proof}

\begin{prop}\label{1152+2016}
There is no finite quotient of $\mathbb{T} (2,3,8)$ of order 1152 or 2016.
\end{prop}

\begin{proof} Assume that $G$ is a group of order 1152 or 2016 admitting a
surjective homomorphism $\mathbb{T} (2,3,8) \rightarrow G$.

Since $\BT(2,3,8)^{ab}\cong {\mathbb Z}/2{\mathbb Z}$, the abelianization of $G$ is a quotient of ${\mathbb Z}/2{\mathbb Z}$ and since there are no perfect groups of order 1152 or 2016, $G^{ab}
 \cong {\mathbb Z}/2{\mathbb Z}$.

\medskip
\underline{$|G|=1152$.} The following MAGMA computation
\begin{verbatim}
> for G in SmallGroups(1152) do
for> if #AbelianQuotient(G) eq 2 then
for|if> if ExSphGens(G,{2,3,8}) then
for|if|if> print G;
for|if|if> end if; end if; end for;
Warning:  May return more than 100,000 groups -- this will take a
VERY long time.  Would a SmallGroupProcess be more appropriate?
>
\end{verbatim}
shows, that $G$ can't have order 1152.

\medskip
\underline{$|G|=2016$.}
Since $G^{ab} \cong \BT(2,3,8)^{ab}$,
$[\BT(2,3,8),\BT(2,3,8)]\cong \mathbb{T}(3,3,4)$ surjects onto $[G,G]$ and therefore  $[G,G]$
is a group of order 1008 admitting an appropriate orbifold homorphism from 
$\BT(3,3,4)$.

Since $\BT(3,3,4)^{ab} \cong {\mathbb Z}/3{\mathbb Z}$ and since there
are no perfect groups of order 1008, we get a contradiction running the
following script.
\begin{verbatim}
> for G in SmallGroups(1008) do
for> if #AbelianQuotient(G) eq 3 then
for|if> if ExSphGens(G,{3,3,4}) then
for|if|if>  print G;
for|if|if>  end if; end if; end for;
>
\end{verbatim}
\end{proof}

\begin{prop}\label{1536}
There is exactly one group $G$ of order $1536$ admitting an appropriate orbifold homomorphism  $\BT(2,3,8) \rightarrow G$.

There is no product-quotient surface of general type with $p_g=0$ with group $G$, whose quotient model has
 $\left\{\frac13(1,1), \frac13(1,2)\right\}$ as basket of singularities.
\end{prop}

\begin{proof}
There are 408641062 groups of order 1536. We have to use
a "SmallGroupProcess" to deal with this case.

The claim follows, running the subsequent MAGMA script
\begin{verbatim}
> P:=SmallGroupProcess(1536);
> i:=1;
> repeat
repeat> G:=Current(P);
repeat> if #AbelianQuotient(G) eq 2 then
repeat|if> if ExSphGens(G,{2,3,8}) then
repeat|if|if> print i;
repeat|if|if> end if;
repeat|if> end if;
repeat> i:=i+1;
repeat> Advance(~P);
repeat> until IsEmpty(P);
408544637
> G:=SmallGroup(1536,408544637);
> FindSurfaces({*1/3,2/3*},{*{2,3,8}^^2*},G);
{@ @}
>
\end{verbatim}
\end{proof}

\begin{prop}\label{768}
1) There is exactly one group $G$ of order $768$ admitting an appropriate orbifold homomorphism
$\BT(2,3,8) \rightarrow G$. 

\noindent
2) $G$ does not admit an appropriate orbifold homomorphism
$\BT(2,4,6) \rightarrow G$.

\noindent
3) There is no product-quotient surface of general type with $p_g=0$ with group $G$, whose quotient model has
 $\left\{2 \times \frac13(1,1), 2 \times \frac13(1,2)\right\}$ as basket of singularities.
\end{prop}

\begin{proof}
Using the same arguments as in the previous case, the two assertions  follow from the following MAGMA computation:
\begin{verbatim}
> P:=SmallGroupProcess(768);
> repeat
repeat>  G:=Current(P);
repeat>   if #AbelianQuotient(G) eq 2 then
repeat|if>    if ExSphGens(G,{2,3,8}) then
repeat|if|if>     print IdentifyGroup(G);
repeat|if|if>    end if;
repeat|if>   end if;
repeat>  Advance(~P);
repeat> until IsEmpty(P);
<768, 1085341>
> G:=SmallGroup(768,1085341);
> ExSphGens(G,{2,4,6});
false
>  FindSurfaces({*1/3^^2,2/3^^2*},{* {2,3,8}^^2 *},G);
{@ @}
>
\end{verbatim}
\end{proof}

Propositions \ref{237}, \ref{1152+2016}, \ref{1536} and \ref{768} exclude all cases in table \ref{tocheck}.

\section{Rational curves on product-quotient surfaces}\label{rational}

We need to recall diagram \ref{diagram1}:
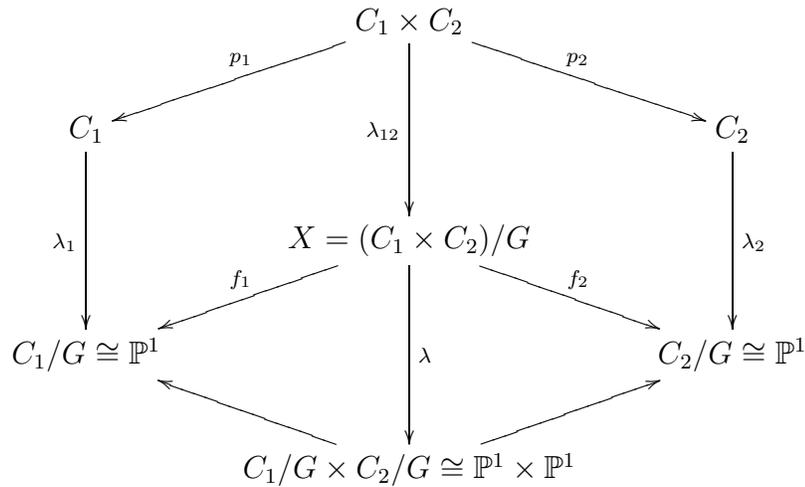
\begin{equation*}
\xymatrix{
&C_1 \times C_2 \ar^{p_2}[dr]\ar_{p_1}[dl]\ar_{\lambda_{12}}[dd]&\\
C_1\ar_{\lambda_1}[dd]&&C_2\ar^{\lambda_2}[dd]\\
&X = (C_1 \times C_2)/G\ar_{f_1}[ld]\ar^{f_2}[dr]\ar^{\lambda}[dd]&\\
C_1/G \cong \PP^1&&C_2/G \cong \PP^1\\
&C_1/G \times C_2/G \cong \PP^1 \times \PP^1 \ar_{}[ul]\ar_{}[ur]&\\
}
\end{equation*}

\medskip

Assume that $\Gamma \subset X$ is a (possibly singular) rational curve. Let $\bar{\Gamma} := \lambda_{12}^*(\Gamma) = \sum_{1}^k n_i \Gamma_i$ be the decomposition in irreducible components of its pull back to $C_1 \times C_2$.

Observe that $n_i=1, \ \forall i$ (since $\lambda_{12}$ has discrete ramification), and that $G$ acts transitively on the set $\{ \Gamma_i | i \in \{1, \ldots ,k \} \}$. Hence there is a subgroup $H \leq G$ of index $k$ acting on $\Gamma_1$ such that $\lambda_{12}(\Gamma_1) = \Gamma_1 /H  = \Gamma$.

Normalizing $\Gamma_1$ and $\Gamma$, we get the following commutative diagram:
\begin{equation}\label{normalization}
\xymatrix{
\tilde{\Gamma}_1\ar[r]\ar_{\gamma}[d]&\Gamma_1\ar[d]\\
\PP^1\ar^{\nu}[r]&\Gamma\\
}
\end{equation}
and, since each automorphism lifts to the normalization, $H$ acts on $\tilde{\Gamma}_1$ and $\gamma$ is the quotient map $\tilde{\Gamma}_1 \rightarrow \tilde{\Gamma}_1 /H \cong \PP^1$.

\begin{lemma}\label{branchpoints}
Let $p$ be a branch point of $\gamma$ of multiplicity $m$. Then $\nu (p)$ is a singular point of $X$ of
type $\frac{1}{n} (1,a)$, where $m | n$.
\end{lemma}

\begin{proof}
 Let $p' \in \tilde{\Gamma}_1$ be a ramification point of $\gamma$ and $g \in H$ a generator of its stabilizer. The stabilizer $A$ of the image of $p'$ in $C_1 \times C_2$ (with respect of the action of $G$) contains $g$, whence $m = o(g)$ divides $n = |A|$.
\end{proof}

\begin{rem}\label{enriques}
It follows from the Enriques-Kodaira classification of complex algebraic surfaces that, if $q(S) = 0$, either
\begin{itemize}
\item[i)] $S$ is rational, or
\item[ii)] $S$ is of general type, or
\item[iii)] $K_S^2\leq 0$.% and if $K_S^2= 0$ then $K_S$ is nef.
\end{itemize}
\end{rem}

\begin{rem}\label{smrat}
On a smooth surface $S$ of general type every irreducible curve $C$ with $K_SC  \leq 0$ is smooth and rational.
\end{rem}

\begin{proof}
Consider the morphism $f \colon S \rightarrow M$ to its
minimal model. Assume that there is an irreducible curve  $C \subset S$ with $K_SC  \leq 0$ which is either
singular or irrational. Then $C$ is not contracted by $f$ and  $C':=f(C)$ is a still singular resp.
irrational curve with $K_MC'\leq K_SC \leq 0$ which, by a classical argument (e.g. cf. \cite{bombieri}, prop. 1), implies
that $C'$ is a smooth rational curve of selfintersection $(-2)$, a contradiction.
\end{proof}

\begin{prop}\label{countingpoints}
Let $S$ be a product-quotient surface of general type. Let $\pi \colon S \rightarrow X$ be the minimal resolution of singularities of the quotient model. Assume that  $\pi_*^{-1}(\Gamma)$
is a $(-1)$-curve in $S$ and let  $x \in \Sing (X)$  be a point of type $\frac{1}{n} (1,a)$, with $\frac{n}{a}=[b_1, \ldots , b_r]$. Consider the map $\nu$ in diagram (\ref{normalization}).
Then
\begin{itemize}
 \item[i)] $ \# \nu^{-1}(x) \leq 1$, if $a = n-1$;
\item[ii)] $ \# \nu^{-1}(x) \leq \sum_{\{b_i \geq 4\}} (b_i-3) + \# \{i : b_i = 3\}$, if $a \neq n-1$.
\end{itemize}
\end{prop}

\begin{proof}
Note that, since we are assuming $\pi_*^{-1}(\Gamma)$ smooth, $\nu=\pi_{|\pi_*^{-1}(\Gamma)}$.

Let $D_i$ be $i$-th curve in the resolution graph of $x$: $D_i$ is smooth, rational with $D_i^2=-b_i$, whence $K_SD_i=b_i-2$. We set $d_i:=D_i \cdot \pi_*^{-1}(\Gamma)$.

After contracting $\pi_*^{-1}(\Gamma)$, $D_i$ maps to $D'_i$ with $KD'_i=KD_i-d_i$. By remark \ref{smrat}, either $D'_i$ is smooth or $KD'_i>0$. In particular: $d_i \leq \max (1,b_i-3)$.

If $b_i = 2$ then $D_i$  intersects $\pi_*^{-1}(\Gamma)$ transversally in at most in one point. Moreover, $\pi_*^{-1}(\Gamma)$ can't intersect two $D_j$ with selfintersection $-2$, since this would produce, after  contracting $\pi_*^{-1}(\Gamma)$, two intersecting  $(-1)$ - curves which is impossible on a surface of general type.

Therefore, if $a = n-1$, $\pi_*^{-1}(\Gamma)$ intersects the whole Hirzebruch-Jung string in at most one
point. This shows part i).

In general,
\begin{multline}
  \# \nu^{-1}(x) \leq \pi_*^{-1}(\Gamma) (\sum D_i) = \\
 = \pi_*^{-1}(\Gamma) (\sum_{\{b_i \geq 4\}}  D_i) + \pi_*^{-1}(\Gamma) (\sum_{\{b_i = 3\}}  D_i) + \pi_*^{-1}(\Gamma) (\sum_{\{b_i =2\}}  D_i) \leq \\
\leq \sum_{\{b_i \geq 4\}} (b_i-3) + \# \{i : b_i = 3\} +1.
\end{multline}
It remains to show that, for $a \neq n-1$, the above inequality cannot be an equality.

In fact, if equality holds, there is an $i$ such that $D_i'$ is a $(-1)$ - curve and
$\forall j \neq i$ we have 
\begin{itemize}
 \item $KD_j'=0$, $D'_j$ is smooth, or
\item $KD_j'=1$, $D'_j$ is singular.
\end{itemize}
$D'_i$ cannot intersect any singular $D'_j$, otherwise the surface obtained after contracting $D'_i$ would violate remark \ref{smrat}. With the same argument we see that $D'_i$ intersects at most one of the smooth $D'_j$. 

In fact, if $r>1$, $D'_i$ intersects exactly one smooth $D'_j$, because a Hirzebruch-Jung string is connected. After the contraction of $D_i'$, $D'_j$ becomes negative with respect to $K$, whence it can be contracted. Recursively, we contract all curves. It follows
that the dual graph of the union of $\pi_*^{-1}(\Gamma)$ with the Hirzebruch-Jung string of the singularity is a tree. By the connectedness of the Hirzebruch-Jung string, $\pi_*^{-1}(\Gamma) (\sum D_i)=1$. Therefore $\forall i$, $b_i=2$ which is equivalent to $a=n-1$.
\end{proof}

The following is an immediate consequence of the above.
\begin{cor}\label{nu}
With the same hypotheses as in prop. \ref{countingpoints} we have:
\begin{itemize}
 \item[i)] $\# \nu^{-1}(\frac{1}{n}(1,1)) \leq \max(1,n-3)$;
\item[ii)] $\# \nu^{-1}(\frac{1}{n}(1,a)) \leq 1$, for $n\leq 7$, $a \neq 1$.
\end{itemize}
\end{cor}
 One possible definition of a rational double point (RDP for short) is the following. For more details we refer to \cite{bpv}.
\begin{definition}
A {\em rational double point} is a singular point of a surface, such that all the exceptional curves of the minimal resolution of it have selfintersection $-2$.
\end{definition}

\begin{prop}\label{baskets}
 Assume that $S$ is a product-quotient surface of general type and assume that the basket of singularities of the quotient model $X$ is one of the following:
\begin{itemize}
 \item[1)] $\{\frac{1}{n}(1,a),\frac{1}{n}(1,n-a)\}$ with either $n \leq 4$ or $n\leq 7$, $1 \neq a< \frac{n}{2}$;
 \item[2)] at most one point $\frac{1}{n}(1,a)$ with either $n\leq 4$ or $n\leq 7$, $a \neq 1$, and RDPs;
 \item[3)] $\{2 \times \frac{1}{3}(1,1)+RDPs\}$, $\{\frac{1}{5}(1,1),\frac{1}{5}(1,4)\}$.
\end{itemize}
Then $S$ is minimal.
\end{prop}

\begin{proof}
 Assume by contradiction that $S$ contains a $(-1)$-curve $E$. Then we can apply prop. \ref{countingpoints} to $\Gamma:= \pi_*(E) \subset X$.

1) In this case, by cor. \ref{nu} and lemma \ref{branchpoints}, $\gamma$ (cf. diagram \ref{normalization}) has at most two critical values, corresponding to the  singular points of $X$. Therefore $\Gamma_1$ is rational, a contradiction.

2) Note that $E$ cannot intersect two distinct $(-2)$ - curves.
In particular, $\Gamma$ can pass through  at most one rational double point, and has to be smooth in this point.

 Therefore, this case is excluded by the same argument as above.

3) We have to treat each basket separately.

\noindent
\underline{$\{2 \times \frac{1}{3}(1,1)+RDPs\}$}: by corollary \ref{nu} $\gamma$ has at most $3$ branch
points, whence by the above argument it has exactly $3$.
Therefore $\Gamma$ passes through both triple points and through one rational double point of $X$: we have found a configuration of
rational curves on $S$ whose dual graph is

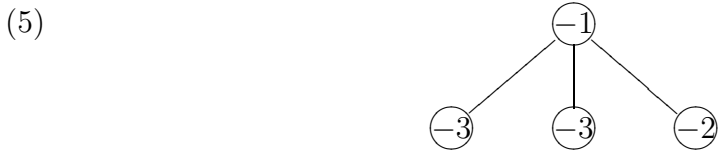
\begin{equation}\label{2x13+rdp}
\xymatrix{
&*+[o][F]{-1}\ar@{-}[d]\ar@{-}[dl]\ar@{-}[dr]&\\
*+[o][F]{-3}&*+[o][F]{-3}&*+[o][F]{-2}\\
}
\end{equation}

\medskip
which cannot occur on a surface of general type because, after contracting the $(-1)$- and the $(-2)$-curve, one gets two intersecting (even tangentially) $(-1)$-curves.

\noindent
\underline{$\{\frac{1}{5}(1,1),\frac{1}{5}(1,4)\}$}: We  get
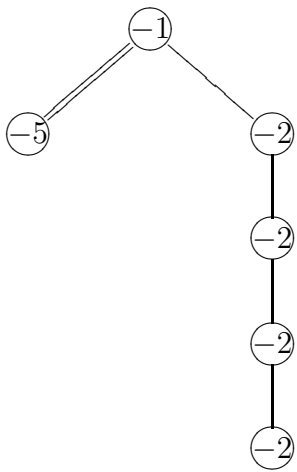
\begin{equation}\label{15+14}
\xymatrix{
&*+[o][F]{-1}\ar@{=}[dl]\ar@{-}[dr]&\\
*+[o][F]{-5}&&*+[o][F]{-2}\ar@{-}[d]\\
&&*+[o][F]{-2}\ar@{-}[d]\\
&&*+[o][F]{-2}\ar@{-}[d]\\
&&*+[o][F]{-2}\\
}
\end{equation}

Contracting $E$ and the whole H-J string coming from the singularity $\frac{1}{5}(1,4)$, the image of the $(-5)$-curve violates rem. \ref{smrat}.
\end{proof}

%\begin{cor}\label{nef}
%Assume that $q(S) = 0$
%and assume that the basket of $X$ is one of the baskets of prop. \ref{baskets}.
%\begin{itemize}
%item[i)] If $K^2_S \leq 0$, then $S$ is not of general type.
%\item[ii)] If $\pi_1(S) \neq \{ 1 \}$, $K^2_S \geq 0$, then $K_S$ is nef.
%\end{itemize}
%\end{cor}
%
%\begin{proof}
%i) follows immediately from prop. \ref{baskets}, since $K^2 > 0$ for a minimal surface of general type.
%
%ii) follows from remark \ref{enriques} and prop. \ref{baskets}.
%\end{proof}

\begin{theo}\label{minimal}
The minimal product-quotient surfaces of general type with $p_g=0$ form $72$ families which are
listed in tables \ref{K2>4} and \ref{K2<3}, and described in appendix \ref{details}.
\end{theo}

\begin{proof}
The case $K^2=8$ has been already classified in \cite{bacat}, \cite{pg=q=0}.

Running our program for $K^2 \in \{7,6,5,4,3,2,1\}$ we have found the surfaces listed in tables
\ref{K2>4} and \ref{K2<3} and one more surface, which we called "the fake Godeaux surface", having  $K^2_S=1$ and
$\pi_1(S)=\ZZ/6\ZZ$ (hence cannot be minimal, cf. \cite{tokyo}).

All the other surfaces are minimal by comparing the baskets appearing in tables \ref{K2>4} and \ref{K2<3} with proposition \ref{baskets} (remembering that $\frac{1}{5}(1,2) = \frac{1}{5}(1,3)$).
\end{proof}

\section{The fake Godeaux surface}\label{fakegodeaux}
Our program produces $73$ families of product-quotient surfaces
of general type with $p_g=0$ and $K^2>0$, and theorem \ref{minimal}shows that $72$ of them are families of minimal surfaces.

The 73rd output in the form of tables \ref{K2>4} and \ref{K2<3} is the following

\begin{table}[hb]
\caption{}
\label{fgodeaux}
\renewcommand{\arraystretch}{1,3}
 \begin{tabular}{|c|c|c|c|c|c|c|c|}
\hline
$K^2_S$&Sing X&$t_1$&$t_2$&$G$&N&$H_1(S,{\mathbb Z})$&$\pi_1(S)$\\
\hline\hline
   1&$1/7, 2/7^2$ &$3^2, 7$&$2,4,7$& $PSL(2,7)$ &1&    ${\mathbb
Z}_6$ &   ${\mathbb Z}_6$ \\
\hline
\end{tabular}
\end{table}

More precisely, the computer gives exactly one pair of appropriate orbifold homomorphisms, which is the following.

We see $G=PSL(2,7)$ as subgroup of ${\mathfrak S}_8$
generated by $(367)(458),(182)(456)$. Then (note that $\BT(a,b,c) \cong \BT(c,b,a)$)
%\begin{align}
\begin{flalign*}
\varphi_1 \colon \BT(7,3,3) & \ra G,&\varphi_2 \colon \BT(7,4,2)& \ra G\\
c_1& \mapsto  (1824375)&c_1 & \mapsto (1658327)\\
c_2& \mapsto  (136)(284)&c_2 & \mapsto (1478)(2653)\\
c_3& \mapsto  (164)(357)&c_3 & \mapsto (15)(23)(36)(47).
\end{flalign*}

As explained in the introduction, choosing three points $p_1,p_2,p_3 \in \PP^1$ (as branch points of $\lambda_1$), three simple loops $\gamma_i$ around them with $\gamma_1\gamma_2\gamma_3 = 1$, $\varphi_1$ determines the monodromy homomorphism and then $C_1$ and $\lambda_1 \colon C_1 \ra C_1/G \cong \PP^1$. Since the covering is determined by the kernel of the monodromy homomorphism, $C_1$ and $\lambda_1$ do not depend on the choice of the loops. 

Since $\Aut(\PP^1)$ is 3-transitive, a different choice of the three branch points will give rise to an isomorphic covering.

Therefore in our situation (up to isomorphism) $C_1$ and $\lambda_1$ are unique.
The same holds $C_2$ and $\lambda_2$. Hence the pair $(\varphi_1, \varphi_2)$ above determines exactly one product-quotient surface $S$, which we have called "the fake Godeaux surface".

Note that by remark \ref{enriques} $S$ is a surface of general type.

This section is devoted to the proof of the following
\begin{theo}
The fake Godeaux surface $S$ has two $(-1)$-curves. Its minimal model has $K^2=3$.
\end{theo}

We first construct two $(-1)-$curves on $S$.

\subsection{The rational curve $E'$}

We can choose the branch points $p_i$ of $\lambda_1$ and $p'_j$ of $\lambda_2$ at our convenience. We set $(p_1, p_2, p_3)=(1,0,\infty)$, $(p'_1, p'_2, p'_3)=(0,\infty,-\frac{9}{16})$.

Consider the normalization $\hat{C}'_1$ of the fibre product between
$\lambda_1$ and the $\ZZ/3\ZZ$-cover $\xi' \colon \PP^1 \rightarrow \PP^1$
defined by $\xi'(t)=t^3$. We have a diagram
\begin{equation*}
\xymatrix{
\hat{C}'_1\ar^{\hat{\xi}'}[r]\ar^{\hat{\lambda}'_1}[d]&C_1\ar^{\lambda_1}[d]\\
\PP^1\ar^{\xi'}[r]&\PP^1 \\
}
\end{equation*}
where the horizontal maps are $\ZZ/3\ZZ$-covers and the vertical maps
are $PSL(2,7)$-covers. Note that $\xi'$ branches on $p_2$, $p_3$ which have branching index 3 for $\lambda_1$:
it follows that $\hat{\xi}'$ is \'etale.

The branch points of $\hat{\lambda}'_1$ are the three points in $\xi'^{-1}(p_1)$, all with branching index 7.

For $C_2$, we take the normalized fibre product between $\lambda_2$ and the map $\eta' \colon
\PP^1 \rightarrow \PP^1$ defined by $\eta'(t)=\frac{(t^3-1)(t-1)}{(t+1)^4}$.

Note that $\eta'$ has degree $4$ and factors
through the involution $t \mapsto \frac1t$. Therefore it is the
composition of two double covers, say $\eta'=\eta_1' \circ \eta_2'$.
We get the following diagram

\begin{equation*}
\xymatrix{
\hat{C}'_2\ar^{\hat{\eta}'}@/^1.7pc/[rr]\ar^{\hat{\eta}'_2}[r]\ar^{\hat{\lambda}'_2}[d]&\bar{C}_2'\ar[d]^{\bar{\lambda}_2'}\ar^{\hat{\eta}'_1}[r]&C_2\ar^{\lambda_2}[d]\\
\PP^1\ar_{\eta'}@/_1pc/[rr]\ar^{\eta'_2}[r]&\PP^1\ar^{\eta'_1}[r]&\PP^1\\
}
\end{equation*}
where the horizontal maps are $\ZZ/2\ZZ$-covers and the vertical maps
are $PSL(2,7)$-covers. 

A straightforward computation shows that $\eta'_1$ branches only on $p'_2$, $p'_3$ and
therefore $\hat{\eta}'_1$ is \'etale. 

The branch points of $\bar{\lambda}'_2$ are the two points in $(\eta'_1)^{-1}(p'_1)$  with branching index 7, and the point $(\eta'_1)^{-1}(p'_2)$ with branching index 2.

A similar computation shows that the branch points of
$\eta'_2$ are $(\eta'_1)^{-1}(p'_2)$ and a point $q'\in (\eta'_1)^{-1}(p'_1)$: $\hat{\eta}'_2$
branches on the 24 points of $(\bar{\lambda}'_2)^{-1}(q')$.

The branch points of $\hat{\lambda}'_2$ are the three points of $(\eta')^{-1}(p'_1)$, each with branching index 7.

\begin{lemma}\label{conjugated}
$(\hat{C}'_1,\hat{\lambda}'_1)$ and $(\hat{C}'_2,\hat{\lambda}'_2)$ are isomorphic as Galois covers of $\PP^1$.
\end{lemma}

\begin{proof}
By construction they have the same group $G=PSL(2,7)$ and the same
branch points, the third roots of $1$, each with branching index
$7$.

We ask the computer for all appropriate orbifold homorphisms $\varphi \colon \BT(7,7,7) \rightarrow PSL(2,7)$ modulo automorphisms (i.e.,  inner automorphisms of $PSL(2,7)$ and  Hurwitz moves, cf. \cite{bacat}). The computer finds two possibilities, returned as the sequence $[\varphi(c_1),\varphi(c_2),\varphi(c_3)]$.
\begin{verbatim}
> FindCurves({* 7^^3 *}, PSL(2,7));
{
    [
        (1, 7, 8, 4, 6, 2, 3),
        (1, 5, 4, 6, 8, 7, 3),
        (1, 2, 6, 8, 4, 5, 3)
    ],
    [
        (1, 5, 8, 2, 3, 4, 6),
        (1, 6, 3, 4, 7, 5, 8),
        (1, 5, 7, 3, 4, 2, 8)
    ]
}
\end{verbatim}

There are two conjugacy classes of elements of order $7$ in $PSL(2,7)$. In both sequences the three entries belong to the same conjugacy class, whereas $(1784623)$ is not conjugate to $(1582346)$. 

We note the following elementary, but crucial fact: let $\varphi \colon \BT \ra G$ be an appropriate orbifold homomorphism such that all $\varphi(c_i)$ belong to the same conjugacy class $\mathcal{C}$. Let $\varphi'$ be an appropriate orbifold homomorphism which is equivalent to $\varphi$ under the equivalence relation generated by inner automorphisms of $G$ and by Hurwitz moves. Then $\varphi'(c_i) \in \mathcal{C}$ for every $i$.

We denote by $\hat{\varphi}_i$ an appropriate orbifold homomorphism associated to $\hat{\lambda}'_i$.

To prove the lemma it suffices now to show that there exist $i, j$ such that $\hat{\varphi}_1(c_i)$ is conjugate to $\hat{\varphi}_2(c_j)$.

By construction, the branch points of $\hat{\lambda}'_1$ are the three points in $\xi'^{-1}(p_1)$, and they are all regular points of $\xi'$. This implies that $\hat{\varphi}_1(c_i)$ is 
conjugate to $\varphi_1(c_1) = (1824375)$. 

Similarly, the branch points of $\hat{\lambda}'_2$ are the three points of $(\eta')^{-1}(p'_1)$, two of them are regular points of $\eta'$. This implies that
two of the $\hat{\varphi}_2(c_i)$ are
conjugate to $\varphi_2(c_1) = (1658327)$, which is conjugate to $(1824375)$.

\end{proof}

Consider the curve $\hat{C}':=\hat{C}'_1=\hat{C}'_2$. By Hurwitz' formula it is a smooth
curve of genus $1+ \frac{168}{2}(-2 + 3\frac67) = 49$ on which we have an action of $G=PSL(2,7)$, an
action of $\ZZ/3\ZZ$, and an action of $\ZZ/2\ZZ$ (given by $\hat{\eta}'_2$). Note that the last two commute with the first (in fact, these two generate
an action of $\mathfrak{S}_3$ on $\hat{C}'$, just look at the induced
action on $\hat{C}'/G =\PP^1$, and how they permute the third roots of
$1$, so we have an explicit faithful action of $PSL(2,7) \times
\mathfrak{S}_3$ on $\hat{C}'$).

We have then a divisor $C':=(\hat{\xi}',\hat{\eta}')(\hat{C}')\subset C_1
\times C_2$ which is $G-$invariant, and the quotient is a
rational curve $\hat{C}'/G \cong \PP^1 \stackrel{e'}{\rightarrow} D'$ contained in the quotient model $X$ of the fake Godeaux surface $S$.

%Explicit computations show that $R'$ has an ordinary triple point at (1,0) (the image of the third roots of $1$) and three ordinary double points
%\begin{itemize}
%\item an ordinary triple point at (1,0) (the image of the third roots of $1$)
%\item an ordinary double point at $(-1,-\frac29)$ (the image of the primitive sixth roots of $1$)
%\item an ordinary double point at the image of $\frac{11+\sqrt{21}}{10} e^\frac{2\pi i}{3}$ and $\frac{11+\sqrt{21}}{10} e^\frac{4\pi i}{3}$
%\item an ordinary double point at the image of $\frac{11-\sqrt{21}}{10} e^\frac{2\pi i}{3}$ and $\frac{11-\sqrt{21}}{10} e^\frac{4\pi i}{3}$
%\end{itemize}

%It follows that $D'$ has ordinary singularities, at most a triple
%point and $3$ double points.

\begin{prop}\label{E'}
$D'$ has an ordinary double point at the singular point
$\frac17(1,1)$, and contains one more singular
point of $X$.

Let $E'$ be the strict transform of
$D'$ on $S$,
let $E_7$ be the exceptional divisor over the singular point of type $\frac17(1,1)$,
$E_2$, $E_4$ be the exceptional divisors over the other singular point
contained in $D'$, with $E_d^2=-d$.

Then $E' E_7=2$, $E' E_4=1$, $E' E_2=0$ and $E'$ is numerically
equivalent to $\pi^*D'-\frac17(2E_7+E_2+2E_4)$. 

Moreover, $E'$ is a smooth rational curve
with selfintersection $-1$.
\end{prop}

\begin{proof}
The composition of $e'$ with $\lambda$ is the map $(\xi',\eta')$, which is birational onto its image.
Therefore $e'$ is also birational, and $D'$ is singular at most
over the singular points of $(\xi',\eta')(\PP^1)=:R'$.

Consider the point $(1,0)\in \PP^1 \times \PP^1$; it is the image of the third roots of
$1$ under the map $(\xi',\eta')$, so $R'$ has a triple point $z$ there.

The points of $\hat{C}'$ lying over $z$
are exactly the $72$ points with nontrivial stabilizer for the action
of $G$, divided in $3$ orbits, one for each branch of the triple point $z$ of $R'$.

Choose a branch, let $P\in \hat{C}'$ be one of the $24$ points in the
corresponding orbit.
Since $\hat{\xi}'$ is \'etale, the map $(\hat{\xi}',\hat{\eta}')$ is a local
diffeomorphism near $P$. $\hat{\lambda}_2'(P)$ is
one of the three branch points of $\hat{\lambda}'_2$
(depending only on the chosen branch of the singular point $z$) one of which is of
ramification for $\eta'_2$, two are not.

In the latter case, both $\hat{\xi}'$, $\hat{\eta}'$ are local
diffeomorphisms, equivariant for the action of the stabilizer of $P$.
It follows that there are local coordinates in $C_1 \times C_2$ such
that the corresponding branch of $C'$ is $\{x=y\}$ and the group acts
as $(x,y) \mapsto (e^\frac{2\pi i}{7} x,e^\frac{2\pi i}{7} y)$:
we have then two branches of $D'$ through the singular point $\frac17(1,1)$.

If instead $\hat{\lambda}_2'(P)$ is a ramification point of $\eta'_2$, the local equation of the branch is $\{x^2=y\}$
and the action is $(x,y) \mapsto (e^\frac{2\pi i}{7} x,
e^\frac{4\pi i}{7} y)$: the corresponding branch of $D'$ passes
through a point $\frac17(1,2)$ and a local computation shows that
its strict transform intersects transversally the $(-4)$-curve and does
not intersect the $(-2)$-curve.

We have computed how $E'$ intersects the $E_d$, the claim on the numerical equivalence follows by standard intersection arguments.

Then \begin{multline*}
K_{C_1 \times C_2} C'=18 \cdot 4 + 3 \cdot 32=168\\
 \Rightarrow K_X \cdot D' = 1 \Rightarrow K_S \cdot E'= K_X \cdot D' -\frac17 (2K_SE_7+K_SE_2+2K_SE_4)=-1
\end{multline*}

Since $S$ is of general type and $E'$ is irreducible with $K_S E' <0$, by remark \ref{smrat}
$E'$ is smooth. This concludes the proof.

%Note that, $C'$ has then $24$ double points, so $p_a(C')=49+24=73$ and $(C')^2=144-168=-24$ so that, as expected
%$$
%(C')^2=-24 \Rightarrow D'^2=-\frac{1}{7} \Rightarrow (E')^2=(D'-\frac{1}{7} (2E_7+E_2+2E_4))^2=-\frac17-\frac47-\frac27=-1
%$$
\end{proof}
\subsection{The rational curve $E''$}

The construction is  similar to the previous one.
We change the choice of the branch points, here $(p_1, p_2, p_3)=
(0,\frac{i}{3\sqrt{3}},-\frac{i}{3\sqrt{3}})$,
$(p'_1, p'_2, p'_3)=(1,\infty,0)$.

We define three maps $\PP^1 \to \PP^1$ as follows: $\xi''_2(t)=\frac{2t}{t^2+1}$,
$\xi''_1(t)=\frac{t^2-1}{t^3-9t}$, $\eta''(t)=t^4$.

Note that $\xi_2''$ is
the quotient by the involution $t \mapsto \frac{1}{t}$, $\xi''_1$ is the $\ZZ/3\ZZ$-cover given by 
$t \rightarrow \frac{t-3}{t+1}$,  and $\eta''$ is the $\ZZ/4\ZZ$-cover given by $t \mapsto it$.

By taking normalized fibre products as in the previous case, we get two commutative diagrams:

\begin{equation*}
\xymatrix{
\hat{C}''_1\ar^{\hat{\xi}''}@/^2pc/[rr]\ar^{\hat{\xi}''_2}[r]\ar^{\hat{\lambda}''_1}[d]&\bar{C}_1''\ar[d]^{\bar{\lambda}_1''}\ar^{\hat{\xi}''_1}[r]&C_1\ar^{\lambda_1}[d]\\
\PP^1\ar_{\xi''}@/_1pc/[rr]\ar^{\xi''_2}[r]&\PP^1\ar^{\xi''_1}[r]&\PP^1\\
}
\end{equation*}

\begin{equation*}
\xymatrix{
\hat{C}''_2\ar^{\hat{\eta}''}[r]\ar^{\hat{\lambda}''_2}[d]&C_2\ar^{\lambda_2}[d]\\
\PP^1\ar^{\eta''}[r]&\PP^1\\
}
\end{equation*}
where the vertical maps
are $PSL(2,7)$-covers and the horizontal maps are cyclic covers.

Note that $\eta''$ branches on $p'_2$, $p'_3$, $\xi_1''$ on $p_2$, $p_3$,
$\xi_2''$ on $\{\pm 1\}  \subset (\xi_1'')^{-1}(p_1)$.

\begin{lemma}\label{conjugated2}
$(\hat{C}''_1,\hat{\lambda}''_1)$ and $(\hat{C}''_2,\hat{\lambda}''_2)$ are isomorphic as Galois cover of $\PP^1$.
\end{lemma}

\begin{proof}
Arguing as in the previous case, we see that $\hat{\lambda}''_1$ is a $PSL(2,7)$-cover with the four branch points $\xi''^{-1}(p_1)$, each of branching index 7, and $\hat{\lambda}''_2$ is a $PSL(2,7)$-cover with the four branch points $\eta''^{-1}(p'_1)$, each of branching index 7. Indeed, $\xi''^{-1}(p_1)=\hat{\lambda}''_2$ is the set of the fourth roots of unity. 

We denote by $\hat{\varphi}_i$ an appropriate orbifold homomorphism associated to $\hat{\lambda}''_i$. Arguing as in the proof of lemma \ref{conjugated} we see that for all $i$, $\hat{\varphi}_1(c_i)$ is conjugate to $\varphi_1(c_1)$ or to $\varphi_1(c_1)^2$. Similarly, $\hat{\varphi}_2(c_i)$ is conjugate to $\varphi_2(c_1)$.
 Since the 
three elements $\varphi_1(c_1)$, $\varphi_1(c_1)^2$ and $\varphi_2(c_1)$ are conjugate in $G$, all $\hat{\varphi}_j(c_i)$ are conjugate to $\varphi_1(c_1) = (1824375)$.

The following computation shows that there
are two equivalence classes of appropriate orbifold homomorphisms $\varphi \colon \BT(7,7,7,7) \ra PSL(2,7)$, distinguished by
the following feature: in one class the $\varphi(c_i)$  are always
pairwise distinct.

\begin{verbatim}
> #FindCurves({* 7^^4 *}, PSL(2,7));
8
> L:={@ @};
> for seq in FindCurves({* 7^^4 *}, PSL(2,7)) do test:= true;
for> for g in seq do
for|for> if not IsConjugate(PSL(2,7),g,PSL(2,7)!(1,8,2,4,3,7,5)) 
         then test:=false; break g;
for|for|if> end if;
for|for> end for;
for> if test then Include(~L,seq);
for|if> end if;
for> end for;
> #L;
2
> for k in [1..#L] do
for> M:={@ @};
for> for seq in HurwitzOrbit(L[k]) do
for|for> for i in [1..3] do for j in [i+1..4] do
for|for|for|for> if seq[i] eq seq[j] then Include(~M, seq);break i;
for|for|for|for|if> end if; end for; end for;end for;
for> #M;
for> end for;
0
840
>
\end{verbatim}
We need to show that $\hat{\varphi}_1$ and $\hat{\varphi}_2$ belong to same class, in fact to the second.

\underline{$\hat{\varphi}_2$}: Consider the map $\eta''$,
and choose as base point for $\pi_1(\PP^1\setminus \{\pm1, \pm i\}, p)$ a point $p=\epsilon$,
$\epsilon \in \RR$, $0<\epsilon<<1$.

We define the following geometric loops with starting point
$\epsilon$: 
\begin{itemize}
\item $\gamma_1$ moves on the real axis from $\epsilon$ to
$1-\epsilon$, then makes a circle counterclockwise around $1$,
and moves back on the real axis to $\epsilon$.
\item $\gamma_2=\alpha(i\gamma_1)\alpha^{-1}$ where $\alpha$ is a quarter of a
circle around $0$ from $\epsilon$ to $i\epsilon$.
\item $\gamma_3=\beta(-\gamma_1)\beta^{-1}$ where $\beta$ is a half
circle around $0$ from $\epsilon$ to $-\epsilon$. 
\item $\gamma_4$ is
a similarly defined loop around $-i$. 
\end{itemize}
Then $\gamma_1\cdots\gamma_4=1$.
Now it is easy to see (since $0$ is a branch point of branching
index $2$ for $\lambda_2$) that the image of $\gamma_1$ in $G$ is the
same as the image of $\gamma_3$ (and the image of $\gamma_2$ is the same as the image $\gamma_4$).

\underline{$\hat{\varphi}_1$}: Let $\gamma_1$, $\gamma_{\infty}$, $\gamma_{-1}$ be geometric loops with
base point $p=0$, $\gamma_j$ around $j$,  $\gamma_1\gamma_{\infty}\gamma_{-1}=1$ in
$\pi_1(\PP^1\setminus\{\pm 1, \infty\})$.
Then we can find geometric loops $\mu_{1}, \mu_{i}, \mu_{-1}, \mu_{-i}$ with base point $0$, $\mu_j$ around $j$,
$\mu_{1}\mu_{i}\mu_{-1}\mu_{-i}=1$ in $\pi_1(\PP^1\setminus\{\pm 1, \pm i\})$,
such that $\xi''_2\circ \mu_i=\gamma_{\infty}$,
$\xi''_2\circ \mu_{-i}=\gamma_1^7\gamma_{\infty}$, which have the same image in $G$.
\end{proof}

The curve $\hat{C}'':=\hat{C}''_1=\hat{C}''_2$ is a smooth
curve of genus $1+ \frac{168}{2}(-2 + 4\frac67) = 121$ with an action of PSL(2,7) (in fact, of $PSL(2,7) \times D_4$, where $D_4$ is the dihedral group of order $8$). The divisor $C'':=(\hat{\xi}'',\hat{\eta}'')(\hat{C}'')\subset C_1
\times C_2$ is $G-$invariant, and the quotient is a
rational curve $\hat{C}''/G \cong \PP^1 \stackrel{e''}{\rightarrow} D''
\subset X$. Note that $\lambda \circ e''=(\xi'',\eta'')$, which is birational.
Therefore $e''$ is also birational.

\begin{prop}
$D''$ has an ordinary double point at the singular point
$\frac17(1,1)$ and contains both the other singular
points of $X$.

Let $E''$ be the strict trasform of
$D''$ on $S$, let $E_7$ be the exceptional divisor over the  singular point of type $\frac17(1,1)$,
$E_2$, $E_4$, $E'_2$, $E'_4$ be the other exceptional divisors,
with $E_d^2=(E'_d)^2=-d$, $E_2E_4=E'_2E'_4=1$.

Then $E'' E_7=2$, $E'' E_4=E'' E'_4=1$, $E'' E_2=E'' E'_2=0$ and $E''$ is numerically
equivalent to $D''-\frac17(2E_7+E_2+2E_4+E'_2+2E'_4)$. 

Moreover, $E''$ is a smooth rational curve
with selfintersection $-1$.
\end{prop}

\begin{proof}
Consider the point $(0,1)$, quartuple point of $R'$. The points of $\hat{C}''$ dominating it
are exactly the $96$ points with nontrivial stabilizer for the action
of $G$, divided in $4$ orbits, one for each branch of $R'$.

Choose a branch of the quartuple point, let $P\in \hat{C}''$ be one of the $24$ points in the
corresponding orbit above it.
Since $\hat{\eta}''$ is \'etale, the map $(\hat{\xi}'',\hat{\eta}'')$
is a local diffeomorphism near $P$. $\hat{\lambda}_1''(P)$ is
one of the branch points of $\hat{\lambda}''_1$
(depending only on the choosen branch), two of which are
ramification points of $\xi''_2$, two are not.

In the latter case, arguing as in the proof of  proposition \ref{E'}, the corresponding
branch of $D''$ passes through the singular point of type $\frac17(1,1)$ and is smooth there. Instead,
in the first case, it passes through a point of type $\frac17(1,2)$, and
its strict transform intersects transversally the $(-4)$-curve and does
not intersect the $(-2)$-curve.

It follows that $E'' E_7 = 2$, $E''  E_2 = E''  E'_2$, $E''(E_4+E'_4)=2$. We still do not know whether $D''$ passes through both singular points $\frac 17(1,2)$ (equivalently $E'' E_4 = E''E'_4=1$), or misses one of them and passes twice through the other.

Then \begin{multline*}
K_{C_1 \times C_2} C''=18 \cdot 4 + 6 \cdot 32=264\\
 \Rightarrow K_X \cdot D'' = \frac{11}{7} \Rightarrow K_S \cdot E''=
\frac{11}{7} -2 \frac57  -2 \frac47 =-1.
\end{multline*}

Since $S$ is of general type and $E''$ is irreducible with $K_S E'' <0$, by remark \ref{smrat},
$E''$ is smooth. This proves that $E''$ is a rational $(-1)$-curve.

If $E'' E_4=2$ or $E''E_4' =2$, after contracting $E''$, we get a contradiction to remark \ref{smrat}.

Therefore, $E'' E_4=E''E_4' =1$.
\end{proof}

\begin{cor}
Let $\pi' \colon S \ra S'$ be the blow down of $E'$ and $E''$. Then $S'$ is minimal.
\end{cor}

\begin{proof}
Since $S$ is of general type, $E'$ and $E''$ are disjoint.  We have a configuration of rational curves on $S$, whose dual graph is the following:
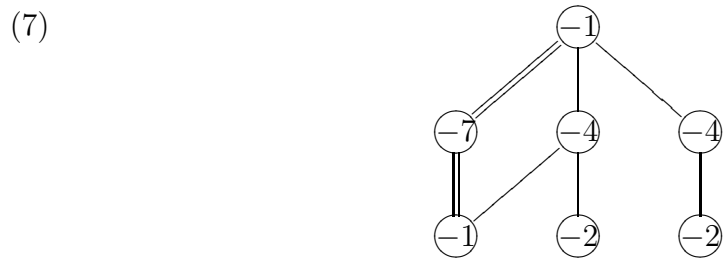
\begin{equation}\label{configS'}
\xymatrix{
&*+[o][F]{-1}\ar@{-}[d]\ar@{=}[dl]\ar@{-}[dr]&\\
*+[o][F]{-7}\ar@{=}[d]&*+[o][F]{-4}\ar@{-}[d]\ar@{-}[dl]&*+[o][F]{-4}\ar@{-}[d]\\
*+[o][F]{-1}&*+[o][F]{-2}&*+[o][F]{-2}\\
}
\end{equation}

After contracting  $E'$ and $E''$ the induced configuration consists of a singular $(-3)$-curve, three smooth $(-2)$-curves and a smooth $(-3)$-curve. 

Assume that there is a smooth rational curve $E'''$ with selfintersection $(-1)$ on $S'$. By the same arguments as in prop. \ref{countingpoints} $E'''$ can intersect only one of the $(-2)$-curves (with multiplicity one), and the smooth (-3)-curve (again with multiplicity one). The singular $(-3)$ curve instead cannot intersect $E'''$, because this would (after contracting $E'''$) give a contradiction to remark \ref{smrat}.

Let $\Gamma \subset X$ be the rational curve $\pi_*\pi'^{-1}_*E'''$ on X. Then, by lemma \ref{branchpoints} the induced map $\gamma$ (cf. diagram \ref{normalization}) has at most two critical values, and therefore $\Gamma_1$ is rational, a contradiction (cf. proof of prop. \ref{baskets}).
\end{proof}

\section{Some remarks about the computational complexity}
In this short section we will comment on the necessity to use a computer algebra program in this paper, and also on the time and memory that is needed for the various calculation we did.

We use the computer algebra program MAGMA, but  our algorithms can be  implemented in any other computer algebra program which has a database of finite groups (e.g. GAP4).

The heaviest computational problem we encountered are caused by the first step of the algorithm in section \ref{algorithm}: there we compute for each $1 \leq K^2 \leq 7$ the possible baskets of singularities for $(K^2,1)$.
Our  algorithm is quite slow, but has a very low memory usage. Indeed, making the algorithm quicker had disastrous effects on the memory usage. 

In the following table we report the computation time and memory usage of the script {\it Baskets} for 
each $K^2$. Almost all computations have been done on a simple workstation with 4GB of RAM.

\medskip
\noindent
\begin{tabular}{|c|c|c|c|c|c|c|c|c|}
\hline
$K^2$&7&6&5&4&3&2&1&0\\
\hline
time (s)&0.00&0.00&0.01&0.11&2.16&45.99&1185.85&43316.7\\
\hline
memory (MB)&7.64&7.64&7.64&7.64&7.64&8.31&9.83&18.42\\
\hline
\end{tabular}

\medskip\noindent

An improvement of this algorithm (i.e., to make it faster without substantially increasing the memory 
usage) would constitute the major step towards extending the results of the present paper to negative 
values of $K^2$. Indeed we give the analogous table for {\it ExistingSurfaces}.

\medskip
\noindent
\begin{tabular}{|c|c|c|c|c|c|c|c|c|}
\hline
$K^2$&7&6&5&4&3&2&1&0\\
\hline
time (s)&0.00&1811&3659&5132&385&8065&2632&84989\\
\hline
memory (MB)&7.64&119.03&119.03&119.55&118.64&120.06&120.23&397.39\\
\hline
\end{tabular}

\medskip\noindent
{\it ExistingSurfaces} first runs {\it Baskets}. We notice that for the first cases the time requested by this first computation is negligible. For $K^2 =0$ it is more or less the half.

The other scripts in the main algorithm are quite harmless in time and memory usage.

\medskip
The computations in sections \ref{exceptional}, \ref{fakegodeaux} are neither time nor memory demanding, except for propositions  \ref{1536} and \ref{768}, where we had to use  a SmallGroup Process (and they actually are quite heavy). In fact, we had to run those two computations on a better workstation (32GB of RAM).

The first lasted 192261.54 sec. (53-54 hours) and needed 19102.22 MB, the second lasted 4581.27 sec. and needed 4509.09 MB.

%%%%%%%%%%%%%%%%%%%%%%%%%%%%%%%%%%%%%%%%%%%%%%%%%%%%%%

\appendix
\section{The minimal product-quotient surfaces of general type with $p_g=0$ and $K^2<8$}\label{details}

In this section we describe all the minimal product-quotient surfaces we have listed in tables 
\ref{K2>4} and \ref{K2<3}, with the exception of the one whose singular model $X$ has at worse 
canonical singularities (these are already described in \cite{pg=q=0} and \cite{4names}).

In the sequel we will follow the scheme below:

\begin{itemize}
\item[$G$:] here we write the group $G$ (most of the times as permutation group);
\item[$t_i$:] here we specify the respective types of the pair of spherical generators of
the group $G$;
\item[$S_1$:] here we list the first set of spherical generators;
\item[$S_2$:]  here we list the second set of spherical generators;
\item[$H_1$:] the first homology group of the surface;
\item[$\pi_1$:] the fundamental group of the surface;
\end{itemize}

\subsection{$K^2=5$, basket $\{ \frac13(1,1)+\frac13(1,2)\}$}

\subsubsection{Group ${\mathfrak S}_4 \times \ZZ_2$:}
\begin{itemize}
\item[$G$:] $\langle (12), (13), (14), (56) \rangle < {\mathfrak S}_6$;
\item[$t_i$:] $(3,2^4)$ and $(6,4,2)$;
\item[$S_1$:] (134),(34)(56),(13)(24)(56),(23)(56), (13)(24)(56);
\item[$S_2$:] (234)(56), (4321)(56), (14);
\item[$H_1$:] $\ZZ_2^2 \times \ZZ_4$;
\item[$\pi_1$:] the fundamental group of this surface fits in two exact sequences
$$
1 \rightarrow \ZZ^2 \rightarrow \pi_1 \rightarrow D_{2,8,3}
\rightarrow 1
$$
$$
1 \rightarrow \ZZ^2 \rightarrow \pi_1 \rightarrow Q(16)
\rightarrow 1
$$
where $Q(16)$ is the generalized quaternion group of order $16$.

The normal subgroups of index $16$ of $\pi_1$ on the left have minimal index
among the normal subgroups of $\pi_1$ with free abelianization.  Let us
recall that $D_{2,8,3}$ is the group $<x,y|x^2,y^8,xyx^{-1}y^{-3}>$ and $Q(16)$ is the group $<x,y|x^8,x^4y^{-2},yxy^{-1}x>$.
\end{itemize}

\subsubsection{Group ${\mathfrak S}_4$.}
\begin{itemize}
\item[$G$:] ${\mathfrak S}_4$;
\item[$t_i$:] $(3, 2^4)$ and $(4^2,3)$;
\item[$S_1$:] $(124),(23),(24),(14),(13)$; 
\item[$S_2$:] $(1243),(1234),(123)$;
\item[$H_1$:] $\ZZ_2^2 \times \ZZ_8$;
\item[$\pi_1$:] the fundamental group of this surface fits in an exact sequence
$$
1 \rightarrow \ZZ^2 \rightarrow \pi_1 \rightarrow \ZZ_8 \rightarrow 1
$$
and the normal subgroup of index $8$ of $\pi_1$ on the left has minimal index
among the normal subgroups of $\pi_1$ with free abelianization.
\end{itemize}

\subsubsection{Group ${\mathfrak S}_4 \times \ZZ_2$:} 
\begin{itemize}
\item[$G$:] $\langle (12), (13), (14), (56) \rangle < {\mathfrak S}_6$;
\item[$t_i$:] $(3,2^3)$ and $(6,4^2)$;
\item[$S_1$:] $(143),(12),(24)(56), (12)(34)(56)$;
\item[$S_2$:] $(134)(56), (1342)(56), (1234)$;
\item[$H_1$:] $\ZZ_2^2 \times \ZZ_8$;
\item[$\pi_1$:] the fundamental group  fits in an exact sequence
$$
1 \rightarrow \ZZ^2 \rightarrow \pi_1 \rightarrow \ZZ_8  \rightarrow 1
$$
and the normal subgroup of index $8$ of $\pi_1$ on the left has minimal index
among the normal subgroups of $\pi_1$ with free abelianization.
\end{itemize}

\subsubsection{Group ${\mathfrak S}_5$.}
\begin{itemize}
\item[$G$:] ${\mathfrak S}_5$;
\item[$t_i$:] $(6, 5, 2)$ and $(4^2,3)$;
\item[$S_1$:] $(13)(245),(14253),(34)$; 
\item[$S_2$:] $(4321),(1534),(235)$;
\item[$H_1$:] $\ZZ_8$;
\item[$\pi_1$:] $D_{8,5,-1}=<x,y|x^8,y^5,xyx^{-1}y>$.
\end{itemize}

\subsubsection{Group ${\mathfrak A}_5$.}
\begin{itemize}
\item[$G$:] ${\mathfrak A}_5$;
\item[$t_i$:] $(3, 2^3)$ and $(5^2,3)$;
\item[$S_1$:] $(152),(14)(23),(23)(45),(14)(25)$; 
\item[$S_2$:] $(15423),(13425),(254)$;
\item[$H_1$:] $\ZZ_2 \times \ZZ_{10}$;
\item[$\pi_1$:] $\ZZ_5 \times Q_8$, where $Q_8$ is the quaternion group $<x,y|x^4,x^2y^{-2},xyx^{-1}y>$.
\end{itemize}

\subsubsection{Group $\ZZ_2^4 \rtimes {\mathfrak S}_3$:} 
this is the semidirect product
obtained by letting (12) and (123) act on $\ZZ_2^4$ respectively as $\begin{pmatrix}
0&1\\
1&0
\end{pmatrix}
\oplus
\begin{pmatrix}
0&1\\
1&0
\end{pmatrix}$
and
$\begin{pmatrix}
0&1\\
1&1
\end{pmatrix}
\oplus
\begin{pmatrix}
0&1\\
1&1
\end{pmatrix}$

\begin{itemize}
\item[$G$:] 
$<x_1,x_2,x_3,x_4,y_2,y_3|x_i^2,y_i^i,[x_i,x_j],
(y_2y_3)^2, y_2x_{2i-1}y_2x_{2i},y_3^{-1}x_{2i-1}y_3x_{2i},$ \\$y_3^{-1}x_{2i}y_3x_{2i-1}x_{2i}>$;

\item[$t_i$:] $(3, 2^3)$ and $(4^2,3)$;
\item[$S_1$:] $y_3x_1, y_2y_3^2x_3, x_1x_3x_4, y_2y_3x_4$; 
\item[$S_2$:] $y_2x_1x_4,y_2y_3x_1,y_3^2x_1x_3$;
\item[$H_1$:] $\ZZ_2 \times \ZZ_8$;
\item[$\pi_1$:] the fundamental group of this surface fits in an exact sequences
$$
1 \rightarrow H \rightarrow \pi_1 \rightarrow D_{8,4,3}
\rightarrow 1
$$
where $H$ is a group with a complicated presentation whose abelian quotient is trivial.
We conjecture $H=\{1\}$ and $\pi_1=D_{8,4,3}$ but the computer could not solve the problem.
Recall that $D_{8,4,3}$ is the group $<x,y|x^8,y^4,xyx^{-1}y^{-3}>$.
\end{itemize}

\subsubsection{Group ${\mathfrak A}_5$.}
\begin{itemize}
\item[$G$:] ${\mathfrak A}_5$;
\item[$t_i$:] $(3, 2^3)$ and $(5^2,3)$;
\item[$S_1$:] $(152),(14)(23),(23)(45),(14)(25)$; 
\item[$S_2$:] $(14235),(15243),(123)$;
\item[$H_1$:] $\ZZ_2 \times \ZZ_{10}$;
\item[$\pi_1$:] $\ZZ_2 \times \ZZ_{10}$.
\end{itemize}

\subsection{$K^2=4$, basket $\{ 2 \times \frac15(1,2)\}$}

\subsubsection{Group ${\mathfrak A}_5$.}
\begin{itemize}
\item[$G$:] ${\mathfrak A}_5$;
\item[$t_i$:] $(5, 2^3)$ and $(5,3^2)$;
\item[$S_1$:] $(13245),(12)(34),(15)(23),(14)(35)$; 
\item[$S_2$:] $(13542),(123),(345)$;
\item[$H_1$:] $\ZZ_2 \times \ZZ_{6}$;
\item[$\pi_1$:] $\ZZ_2 \times \ZZ_{6}$.
\end{itemize}

\subsubsection{Group $\ZZ_2^4 \rtimes D_5$:} this is the semidirect product
obtained by letting a symmetry and a rotation of $D_5$ act on $\ZZ_2^4$ respectively as $\begin{pmatrix}
1&0&1&1\\
1&1&1&0\\
0&0&0&1\\
0&0&1&0
\end{pmatrix}$
and
$\begin{pmatrix}
1&0&0&1\\
1&1&0&0\\
0&1&1&0\\
0&0&1&0
\end{pmatrix}$. 

\begin{itemize}
\item[$G$:] 
\begin{eqnarray*}
<x_1,x_2,x_3,x_4,y_2,y_5|&x_i^2,y_i^i,[x_i,x_j],(y_2y_5)^2,\\
&y_2x_1y_2x_1x_2,y_2x_2y_2x_2,y_2x_3y_2x_1x_2x_4,y_2x_4y_2x_1x_3,\\
&y_5^{-1}x_1y_5x_1x_2,y_5^{-1}x_2y_5x_2x_3,y_5^{-1}x_3y_5x_3x_4,y_5^{-1}x_4y_5x_1>
\end{eqnarray*}
\item[$t_i$:] $(5, 4^2)$ and $(5, 4,2)$;
\item[$S_1$:] $y_5^2x_1,y_2y_5^2x_2x_4,y_2x_4$; 
\item[$S_2$:] $y_5x_2x_3,y_2y_5x_1x_2x_3x_4,y_2x_1x_3x_4$;
\item[$H_1$:] $\ZZ_8$;
\item[$\pi_1$:] the fundamental group fits in an exact sequences
$$
1 \rightarrow H \rightarrow \pi_1 \rightarrow \ZZ_8
\rightarrow 1.
$$

where $H$ is a group with a complicated presentation whose abelian quotient is trivial.
We conjecture $H=\{1\}$ and $\pi_1=\ZZ_8$ but the computer could not solve the problem.
\end{itemize}

\subsubsection{Group $\ZZ_2^4 \rtimes D_5$:} 

\begin{itemize}
\item[$G$:] as above;
\item[$t_i$:] $(5, 4^2)$ and $(5, 4,2)$;
\item[$S_1$:] $y_5^3x_1x_4,y_2x_3,y_2y_5^2x_2x_4$; 
\item[$S_2$:] $y_5^4x_1x_2x_3,y_2x_2x_4,y_2y_5$;
\item[$H_1$:] $\ZZ_8$;
\item[$\pi_1$:] the fundamental group fits in an exact sequences
$$
1 \rightarrow H \rightarrow \pi_1 \rightarrow \ZZ_8
\rightarrow 1.
$$

where $H$ is a group with a complicated presentation whose abelian quotient is trivial.
We conjecture $H=\{1\}$ and $\pi_1=\ZZ_8$ but the computer could not solve the problem.
\end{itemize}

\subsubsection{Group $\ZZ_2^4 \rtimes D_5$:} 
\begin{itemize}
\item[$G$:] as above;
\item[$t_i$:] $(5, 4^2)$ and $(5, 4,2)$;
\item[$S_1$:] $y_5^3x_1x_4,y_2x_3,y_2y_5^2x_2x_4$; 
\item[$S_2$:] $y_5x_2x_3,y_2y_5x_1x_2x_3x_4,y_2x_1x_3x_4$;
\item[$H_1$:] $\ZZ_8$;
\item[$\pi_1$:] the fundamental group fits in an exact sequences
$$
1 \rightarrow H \rightarrow \pi_1 \rightarrow \ZZ_8
\rightarrow 1.
$$

where $H$ is a group with a complicated presentation whose abelian quotient is trivial.
We conjecture $H=\{1\}$ and $\pi_1=\ZZ_8$ but the computer could not solve the problem.
\end{itemize}

\subsubsection{Group ${\mathfrak A}_6$.}
\begin{itemize}
\item[$G$:] ${\mathfrak A}_6$;
\item[$t_i$:] $(5,4,2)$ and $(5,3^2)$;
\item[$S_1$:] $(14623),(13)(2564),(12)(56)$; 
\item[$S_2$:] $(14562),(134)(265),(243)$;
\item[$H_1$:] $\ZZ_{6}$;
\item[$\pi_1$:] $\ZZ_{6}$.
\end{itemize}

\subsection{$K^2=3$, basket $\{ \frac15(1,1)+\frac15(1,4)\}$}

\subsubsection{Group ${\mathfrak A}_5$.}
\begin{itemize}
\item[$G$:] ${\mathfrak A}_5$;
\item[$t_i$:] $(5, 2^3)$ and $(5,3^2)$;
\item[$S_1$:] $(14235),(23)(45),(13)(45),(14)(35)$; 
\item[$S_2$:] $(13542),(123),(345)$;
\item[$H_1$:] $\ZZ_2 \times \ZZ_{6}$;
\item[$\pi_1$:] $\ZZ_2 \times \ZZ_{6}$.
\end{itemize}

\subsubsection{Group $\ZZ_2^4 \rtimes D_5$:} this is the semidirect product
obtained by letting a symmetry and a rotation of $D_5$ act on $\ZZ_2^4$ respectively as $\begin{pmatrix}
1&0&1&1\\
1&1&1&0\\
0&0&0&1\\
0&0&1&0
\end{pmatrix}$
and
$\begin{pmatrix}
1&0&0&1\\
1&1&0&0\\
0&1&1&0\\
0&0&1&0
\end{pmatrix}$. 

\begin{itemize}
\item[$G$:] \begin{eqnarray*}
<x_1,x_2,x_3,x_4,y_2,y_5|&x_i^2,y_i^i,[x_i,x_j],(y_2y_5)^2,\\
&y_2x_1y_2x_1x_2,y_2x_2y_2x_2,y_2x_3y_2x_1x_2x_4,y_2x_4y_2x_1x_3,\\
&y_5^{-1}x_1y_5x_1x_2,y_5^{-1}x_2y_5x_2x_3,y_5^{-1}x_3y_5x_3x_4,y_5^{-1}x_4y_5x_1>;
\end{eqnarray*}

\item[$t_i$:] $(5, 4^2)$ and $(5, 4,2)$;
\item[$S_1$:] $y_5^2x_1,y_2y_5^2x_2x_4,y_2x_4$; 
\item[$S_2$:] $y_5^3x_1x_3,y_2y_5^3x_4,y_2x_1x_3x_4$;
\item[$H_1$:] $\ZZ_8$;
\item[$\pi_1$:] the fundamental group fits in an exact sequences
$$
1 \rightarrow H \rightarrow \pi_1 \rightarrow \ZZ_8
\rightarrow 1.
$$

where $H$ is a group with a complicated presentation whose abelian quotient is trivial.
We conjecture $H=\{1\}$ and $\pi_1=\ZZ_8$ but the computer could not solve the problem.
\end{itemize}

\subsubsection{Group $\ZZ_2^4 \rtimes D_5$:} 

\begin{itemize}
\item[$G$:] as above;
\item[$t_i$:] $(5, 4^2)$ and $(5, 4,2)$;
\item[$S_1$:] $y_5^3x_1x_4,y_2x_3,y_2y_5^2x_2x_4$; 
\item[$S_2$:] $y_5^3x_2x_4,y_2y_5^2x_1x_4,y_2y_5^4x_1x_2x_4$;
\item[$H_1$:] $\ZZ_8$;
\item[$\pi_1$:] the fundamental group fits in an exact sequences
$$
1 \rightarrow H \rightarrow \pi_1 \rightarrow \ZZ_8
\rightarrow 1.
$$
where $H$ is a group with a complicated presentation whose abelian quotient is trivial.
We conjecture $H=\{1\}$ and $\pi_1=\ZZ_8$ but the computer could not solve the problem.
\end{itemize}

\subsubsection{Group $\ZZ_2^4 \rtimes D_5$:} 

\begin{itemize}
\item[$G$:] as above;
\item[$t_i$:] $(5, 4^2)$ and $(5, 4,2)$;
\item[$S_1$:] $y_5^3x_1x_4,y_2x_3,y_2y_5^2x_2x_4$; 
\item[$S_2$:] $y_5^3x_1x_3,y_2y_5^3x_4,y_2x_1x_3x_4$;
\item[$H_1$:] $\ZZ_8$;
\item[$\pi_1$:] the fundamental group fits in an exact sequences
$$
1 \rightarrow H \rightarrow \pi_1 \rightarrow \ZZ_8
\rightarrow 1.
$$
where $H$ is a group with a complicated presentation whose abelian quotient is trivial.
We conjecture $H=\{1\}$ and $\pi_1=\ZZ_8$ but the computer could not solve the problem.
\end{itemize}

\subsubsection{Group ${\mathfrak A}_6$.}
\begin{itemize}
\item[$G$:] ${\mathfrak A}_6$;
\item[$t_i$:] $(5, 4,2)$ and $(5,3^2)$;
\item[$S_1$:] $(14623),(13)(2564),(12)(56)$; 
\item[$S_2$:] $(15342),(164),(135)(246)$;
\item[$H_1$:] $\ZZ_{6}$;
\item[$\pi_1$:] $\ZZ_{6}$.
\end{itemize}

\subsection{$K^2=3$, basket $\{ 2 \times \frac12(1,1) + \frac13(1,1) +\frac13(1,2)\}$}
\subsubsection{Group ${\mathfrak S}_4 \times \ZZ_2$:} 

\begin{itemize}
\item[$G$:] $\langle (12), (13), (14), (56) \rangle < {\mathfrak S}_6$;
\item[$t_i$:] $(4, 3, 2^2)$ and $(6,4,2)$;
\item[$S_1$:] $(1234),(234),(13)(24)(56),(34)(56)$;
\item[$S_2$:] $(234)(56),(4321)(56),(14)$;
\item[$H_1$:] $\ZZ_2 \times \ZZ_{4}$;
\item[$\pi_1$:] $\ZZ_2 \times \ZZ_{4}$.
\end{itemize}

\subsection{$K^2=2$, basket $\{ 2 \times \frac13(1,1) + 2 \times \frac13(1,2)\}$}

\subsubsection{Group $ {\mathfrak A}_4 \times {\mathbb Z}_2$:}
\begin{itemize}
\item[$G$:] $\langle (123),(12)(34),(56) \rangle < {\mathfrak S}_6$;
\item[$t_i$:] $(6^2,2)$ and $(3^2, 2^2)$;
\item[$S_1$:] $(132)(56),(142)(56),(13)(24)$;
\item[$S_2$:] $(234), (123), (13)(24)(56),(14)(23)(56)$;
\item[$H_1$:] $\ZZ_2^2$;
\item[$\pi_1$:] $Q_8$.
\end{itemize}

\subsubsection{Group $ {\mathfrak S}_4$:}
\begin{itemize}
\item[$G$:] ${\mathfrak S}_4$;
\item[$t_i$:] $(4^2,3)$ and $(3^2, 2^2)$;
\item[$S_1$:] $(123), (134), (12), (24)$;
\item[$S_2$:] $(1234), (1243), (124)$;
\item[$H_1$:] $\ZZ_8$;
\item[$\pi_1$:] $\ZZ_8$.
\end{itemize}

\subsubsection{Group $ \ZZ_5^2 \rtimes \ZZ_3$:} this is the semidirect product
obtained by letting a generator of $\ZZ_3$ act on $\ZZ_5^2$ as $\begin{pmatrix}
1&1\\
2&3
\end{pmatrix}$. 
\begin{itemize}
\item[$G$:] $\langle x_1,x_2,y| x_i^5,[x_1,x_2],y^3,y^{-1}x_1^{-1}yx_1x_2^2,y^{-1}x_2^{-1}yx_1x_2^3\rangle $;
\item[$t_i$:] both $(5, 3^2)$;
\item[$S_1$:] $x_1^3x_2^2,y^2x_1^3x_2^4,y$;
\item[$S_2$:] $x_1^3,yx_1,y^2x_1^4x_2^2$;
\item[$H_1$:] $\ZZ_5$;
\item[$\pi_1$:] the fundamental group fits in an exact sequences
$$
1 \rightarrow H \rightarrow \pi_1 \rightarrow \ZZ_5
\rightarrow 1.
$$
where $H$ is a group with a complicated presentation whose abelian quotient is trivial.
We conjecture $H=\{1\}$ and $\pi_1=\ZZ_5$ but the computer could not solve the problem.
\end{itemize}

\subsubsection{Group $ \ZZ_5^2 \rtimes \ZZ_3$:} 
\begin{itemize}
\item[$G$:] as above
\item[$t_i$:] both $(5, 3^2)$;
\item[$S_1$:] $x_1^3x_2^2,y^2x_1^3x_2^4,y$;
\item[$S_2$:] $x_1^4x_2^3,yx_1x_2,y^2x_1^4x_2^3$;
\item[$H_1$:] $\ZZ_5$;
\item[$\pi_1$:] the fundamental group fits in an exact sequences
$$
1 \rightarrow H \rightarrow \pi_1 \rightarrow \ZZ_5
\rightarrow 1.
$$
where $H$ is a group with a complicated presentation whose abelian quotient is trivial.
We conjecture $H=\{1\}$ and $\pi_1=\ZZ_5$ but the computer could not solve the problem.
\end{itemize}

\subsubsection{Group ${\mathfrak A}_5$.}
\begin{itemize}
\item[$G$:] ${\mathfrak A}_5$;
\item[$t_i$:] $(5, 3^2)$ and $(3,2^3)$;
\item[$S_1$:] $(13542),(123),(345)$; 
\item[$S_2$:] $(152),(14)(23),(23)(45),(14)(25)$;
\item[$H_1$:] $\ZZ_2^2$;
\item[$\pi_1$:] $\ZZ_2^2$.
\end{itemize}

\subsection{$K^2=2$, basket $\{ 2 \times \frac12(1,1) + \frac14(1,1) + \frac14(1,3)\}$}

\subsubsection{Group $PSL(2,7)$:} 
\begin{itemize}
\item[$G$:] $\langle (34)(56),(123)(457) \rangle < {\mathfrak S}_7$;
\item[$t_i$:] $(7,4,2)$ and $(4, 3^2)$;
\item[$S_1$:] (1436275),(14)(2357),(36)(45);
\item[$S_2$:](1236)(47), (245)(376), (164)(257);
\item[$H_1$:] $\ZZ_3$;
\item[$\pi_1$:] $\ZZ_3$.
\end{itemize}

\subsubsection{Group $PSL(2,7)$:} 
\begin{itemize}
\item[$G$:] $\langle (34)(56),(123)(457) \rangle < {\mathfrak S}_7$;
\item[$t_i$:] $(7,4,2)$ and $(4, 3^2)$;
\item[$S_1$:] (1436275),(14)(2357),(36)(45);
\item[$S_2$:](34)(1675), (164)(257), (134)(265);
\item[$H_1$:] $\ZZ_3$;
\item[$\pi_1$:] $\ZZ_3$.
\end{itemize}

\subsubsection{Group ${\mathfrak A}_6$.}

\begin{itemize}
\item[$G$:] ${\mathfrak A}_6$;
\item[$t_i$:] $(5, 4,2)$ and $(4,3^2)$;
\item[$S_1$:] $(14623),(13)(2564),(12)(56)$; 
\item[$S_2$:] $(16)(2435),(246),(162)(345)$;
\item[$H_1$:] $\ZZ_3$;
\item[$\pi_1$:] $\ZZ_3$.
\end{itemize}

\subsubsection{Group ${\mathfrak A}_6$.}
\begin{itemize}
\item[$G$:] ${\mathfrak A}_6$;
\item[$t_i$:] $(5, 4,2)$ and $(4,3^2)$;
\item[$S_1$:] $(14623),(13)(2564),(12)(56)$; 
\item[$S_2$:] $(1365)(24),(124)(356),(125)$;
\item[$H_1$:] $\ZZ_3$;
\item[$\pi_1$:] $\ZZ_3$.
\end{itemize}

\subsubsection{Group ${\mathfrak S}_5$.}
\begin{itemize}
\item[$G$:] ${\mathfrak S}_5$;
\item[$t_i$:] $(5, 4,2)$ and $(6,4,3)$;
\item[$S_1$:] $(15432),(1235),(45)$; 
\item[$S_2$:] $(15)(234),(2453),(153)$;
\item[$H_1$:] $\ZZ_3$;
\item[$\pi_1$:] $\ZZ_3$.
\end{itemize}

\subsubsection{Group ${\mathfrak S}_5$.}
\begin{itemize}
\item[$G$:] ${\mathfrak S}_5$;
\item[$t_i$:] $(5, 4,2)$ and $(6,4,3)$;
\item[$S_1$:] $(15432),(1235),(45)$; 
\item[$S_2$:] $(14)(235),(1254),(432)$;
\item[$H_1$:] $\ZZ_3$;
\item[$\pi_1$:] $\ZZ_3$.
\end{itemize}

\subsection{$K^2=1$, basket $\{ 4 \times \frac12(1,1) + \frac13(1,1) + \frac13(1,2)\}$}

\subsubsection{Group ${\mathfrak S}_5$.}
\begin{itemize}
\item[$G$:] ${\mathfrak S}_5$;
\item[$t_i$:] $(3,2^3)$ and $(4^2,3)$;
\item[$S_1$:] $(123),(34),(23),(13)(24$; 
\item[$S_2$:] $(1234),(1243),(124)$;
\item[$H_1$:] $\ZZ_4$;
\item[$\pi_1$:] $\ZZ_4$.
\end{itemize}

\subsubsection{Group $PSL(2,7)$:}

\begin{itemize}
\item[$G$:] $\langle (34)(56),(123)(457) \rangle < {\mathfrak S}_7$;
\item[$t_i$:] $(7,3,2)$ and $(4^2, 3)$;
\item[$S_1$:] $(1476532),(164)(235),(26)(47)$;
\item[$S_2$:]$(1765)(23),(17)(3645),(236)(475)$;
\item[$H_1$:] $\ZZ_2$;
\item[$\pi_1$:] $\ZZ_2$.
\end{itemize}

\subsubsection{Group ${\mathfrak S}_4 \times \ZZ_2$:}

\begin{itemize}
\item[$G$:] $\langle (12), (13), (14), (56) \rangle < {\mathfrak S}_6$;
\item[$t_i$:] $(3,2^3)$ and $(6,4,2)$;
\item[$S_1$:] $(134),(13)(24)(56),(23),(24)(56)$;
\item[$S_2$:] $(143)(56),(1234)(56),(23)$;
\item[$H_1$:] $\ZZ_2 $;
\item[$\pi_1$:] $\ZZ_2$.
\end{itemize}

\section{The Magma Script}

\begin{footnotesize}
\begin{verbatim}

// We first need to find, for each K^2, what are the possible baskets of
// singularities. By Lemma 1.8 the sum of the invariants B of the
// singularities must equal 3(8-K^2).
//
// We will represent a singular point 1/n(1,a) by the rational number
// a/n; hence a basket of singularities will be a multiset of rational
// numbers. Remember that cyclic quotient singularities 1/n(1,a) and
// 1/n(1,a') are isomorphic if a*a'=1 mod n, so we must consider rational
// numbers in (0,1) modulo the equivalence relation a/n~a'/n.
//
// The invariant B of a singularity 1/n(1,a) equals (a+a')/n+sum(b_i),
// where b_i are the entries of the continuous fraction of n/a: we see
// them as the sequence [b_1,...,b_r].  Note that the continuous
// fraction of n/a' is the "reversed" sequence [b_r,...,b_1].
//
// This can be seen as a bijection between rational numbers in (0,1)
// and sequences of integers strictly bigger than 1.
// We make this bijiection explicit by the following scripts.

ContFrac:=function(s)
  CF:=[ ]; r:=1/s;
  while not IsIntegral(r) do
    Append(~CF, Ceiling(r)); r:=1/(Ceiling(r)-r);
  end while;
  return Append(CF, r);
end function;

Nq:=func<cf|#cf eq 1 select cf[1] else cf[1]-1/$$(Remove(cf,1))>;

RatNum:=func<seq|1/Nq(seq)>;

// "Wgt" computes the weight of a sequence, i.e., the sum of its
// entries. It bounds strictly from below B of the corresponding
// singular point.

Wgt:=function(seq)
  w:=0; for i in seq do w+:=i; end for; return w;
end function;

// The next script computes all rational number whose continuous
// fraction has small weight, by listing all sequences (modulo
// "reverse") and storing the corresponding rational number.

RatNumsWithSmallWgt:=function(maxW)
  S:={ }; T:={}; setnums:={RationalField()| };
  for i in [2..maxW] do Include(~S, [i]); end for;
  for i in [1..Floor(maxW/2)-1] do
  for seq in S do
    if #seq eq i then
    if maxW-Wgt(seq) ge 2 then
    for k in [2..maxW-Wgt(seq)] do
     Include(~S,Append(seq, k));
    end for; end if; end if;
  end for; end for;
  for seq in S do
  if Reverse(seq) notin T then Include(~T,seq);
  end if; end for;
  for seq in T do Include(~setnums, RatNum(seq)); end for;
  return setnums;
end function;

// The next two scripts compute the invariants B and e of a rational
// number (i.e., of the corresponding singular point).

InvB:=func<r|Wgt(ContFrac(r))+r+RatNum(Reverse(ContFrac(r)))>;

Inve:=func<r|#ContFrac(r)+1-1/Denominator(RationalField()!r)>;

// The next two scripts compute the invariants B and e of a multiset
// of rational numbers (corresponding to a basket of singular points).

InvBSet:= function(basket)
  B:=0; for r in basket do B+:=InvB(r); end for; return B;
end function;

InveSet:= function(basket)
  e:=0; for r in basket do e+:=Inve(r); end for; return e;
end function;

// Here is the invariant k of the basket:

Invk:=func<r|InvBSet(r)-2*InveSet(r)>;

// The next script computes all rational numbers with weight bounded
// from above by maxW, as computed by RatNumsWithSmallWgt, and returns
// them in a sequence ordered by the value of their invariant B,
// starting from the one with biggest B.

OrderedRatNums:=function(maxW)
   seq:=[RationalField()| ]; seqB:=[RationalField()| ];
   set:=RatNumsWithSmallWgt(Floor(maxW));
     for r in set do i:=1;
       for s in seqB do
       if s gt InvB(r) then i+:=1;
       else break s;
       end if; end for;
     Insert(~seq, i, r); Insert(~seqB, i, InvB(r));
     end for;
  return seq;
end function;

// The next one, CutSeqByB, takes a sequence "seq" and recursively
// removes the first element if its invariant B is at least maxB.

CutSeqByB:=function(seq,maxB)
  Seq:=seq;
  while #Seq ge 1 and InvB(Seq[1]) gt maxB do Remove(~Seq,1); end while;
  return Seq;
end function;

// Now we have a way to compute the set of rationals with B bounded by
// the integer maxB, ordered by B:
// CutSeqByB(OrderedRatNums(maxB-1),maxB)
//
// The next script takes a sequence of rationals ordered by B
// and computes the baskets with invariant exactly B that use only these
// rationals.
// The function is as follows:
// -- first remove the elements with B too big to be in a basket
// -- then take the first element, say r, if B(r)=B, store {* r *}
// -- else attach it to each basket with invariant B-B(r)
//    (computed recalling the function with the same sequence)
//     and store the result
// -- now we have all baskets containing r: remove r from the sequence
//    and repeat the procedure until the sequence is empty

BasketsWithSeqAndB:=function(seq,B)
  ratnums:=CutSeqByB(seq,B); baskets:={ };
  while #ratnums gt 0 do
    bigguy:=ratnums[1];
    if InvB(bigguy) eq B then Include(~baskets,{* bigguy *}); else
    for basket in $$(ratnums, B-InvB(bigguy)) do
      Include(~baskets, Include(basket, bigguy));
    end for; end if;
    Remove(~ratnums,1);
  end while;
  return baskets;
end function;

// Now we can compute all Baskets with a given B:

BasketsWithSmallB:=func<B|
                   BasketsWithSeqAndB(OrderedRatNums(Ceiling(B)-1),B)>;

// We do not need all these baskets, since most of them violate the Lemma 1.7.
// The next two scripts take care of this: "TestBasket" will check if a basket
// violates Lemma 1.7; "Basket" will take the output of BasketsWithSmallB and
// removes all the baskets which violate the condition.

TestBasket:=function(basket)
  firstseq:=[];
  for r in basket do Append(~firstseq,r); end for;
  setofseqs:={ firstseq };
  for i in [1..#firstseq] do newseqs:={};
    for seq in setofseqs do
      Include(~newseqs,
             Insert(Remove(seq,i),i,RatNum(Reverse(ContFrac(seq[i])))));
    end for;
    setofseqs:=setofseqs join newseqs;
  end for;
  test:=false;
  for seq in setofseqs do
    if IsIntegral(Wgt(seq)) then test:=true;
    end if;
  end for;
  return test;
end function;

Baskets:=function(B)
  baskets:={ };
  for basket in BasketsWithSmallB(B) do
    if TestBasket(basket) then Include(~baskets, basket);
    end if;
  end for;
  return baskets;
end function;

// Now we have found, for each K^2, a finite and rather small number of
// possible baskets. The next step is to restrict, for each basket, to finitely
// many signatures. We will represent a signature as the multiset of naturals 
// {* m_i *}.
//
// We first define the index of a basket of singularities as the lowest 
// common multiple of the indices of the singularities

GI:=func<r|Denominator(r)/GCD(Numerator(r)+1,Denominator(r))>;

GorInd:= function(bas)
  I:=1; 
  for r in bas do I:=LCM(IntegerRing()!I,IntegerRing()!GI(r)); end for;
  return I;
end function;

// We need moreover the invariant Theta of a signature

Theta:=function(type)
  t:=-2; for n in type do t+:=1-1/n; end for;
  return t;
end function;

// The input of the next program are 4 numbers, CardBasket, Length, SBound and
// HBound (SBound<=HBound), and its output are all signatures with 
// #signature=Length such that (for C:=max(1/6,(Length-3)/2)
// 1) each m_i is smaller than HBound/C;
// 2) most m_i are smaller than SBound/C, the number of exceptions
//    being bounded from above by half of CardBasket.
// For sparing time, the script first checks if the length is smaller
// than the number of possible exceptions to 2, in which case only the
// inequality 1 is to consider.

CandTypes:=function(CardBasket,Length,SBound,HBound)
  C:=Maximum(1/6,(Length-3)/2); S:=Floor(SBound/C); H:=Floor(HBound/C);
  Exc:=Floor(CardBasket/2);
  if Length le Exc then Types:=Multisets({x: x in [2..H]},Length);
  else Types:=Multisets({x: x in [2..S]},Length);
    for k in [1..Exc] do
    for TypeBegin in Multisets({x: x in [2..S]},Length-k) do
    for TypeEnd in Multisets({x: x in [S+1..H]},k) do
      Include(~Types, TypeBegin join TypeEnd);
    end for; end for; end for;
  end if;
  return Types;
end function;

// The next script, ListOfTypesBas, finds all signatures compatible with the
// basket in the input (i.e., which respect Proposition 1.11).
// We use
// 1) Theta<= maxTh:=(K^2+k)/4 (follows from 1.11.a),
// 2) #signature<= 2*Theta+4 (follows from the definition of Theta).

ListOfTypesBas:=function(basket)
   S:={ }; B:=InvBSet(basket); k:=Invk(basket); I:=GorInd(basket);
   Ksquare:=8-B/3; maxTh:=(Ksquare+k)/4;
   for h in [3..Floor(2*maxTh+4)] do
   for cand in CandTypes(#basket,h,maxTh+1,2*I*maxTh+1) do
         T:=Theta(cand);
   if T le maxTh then
   if T gt 0 then Alpha:=maxTh/T;
   if Alpha in IntegerRing() then
   if forall{n : n in cand | 2*Alpha*I/n in IntegerRing()} then bads:=0;
     for n in cand do
     if Alpha/n notin IntegerRing() then bads +:=1;
     end if; end for;
   if bads le #basket/2 then Include(~S,cand);
   end if; end if; end if; end if; end if; end for; end for;
 return S;
end function;

// Finally, we can conlude the second step, by writing a script which
// lists, for given K^2, all possible baskets (by using Baskets) and for
// each basket all possible signatures (by using ListOfTypesBas)

ListOfTypes:=function(Ksquare)
  S:=[* *];
  for basket in Baskets(3*(8-Ksquare)) do L:=ListOfTypesBas(basket);
  if not IsEmpty(L) then Append(~S,[* basket, L *]);
  end if; end for;
 return S;
end function;

// Now we are left with the last step: for each basket we need to
// consider all pairs of possible signatures and look for groups of the
// correct order which have two sets of spherical generators of these
// signatures which give a surface with the prescribed basket of
// singularities. First we need to write some command which is not
// implemented in MAGMA.

// This extracts from a finite group the set of elements of a certain
// order.

ElsOfOrd:=function(group,order)
  Els:={ };
  for g in group do if Order(g) eq order then Include(~Els, g);
  end if; end for;
  return Els;
end function;

// TuplesOfGivenOrder creates a sequence of the same length as the input
// sequence seq, whose entries are subsets of the group in the input,
// and precisely the subsets of elements of order the corresponding
// entry of seq

TuplesOfGivenOrders:=function(group,seq)
  SEQ:=[];
  for i in [1..#seq] do
    if IsEmpty(ElsOfOrd(group,seq[i])) then SEQ:=[]; break i;
    else Append(~SEQ,ElsOfOrd(group,seq[i]));
    end if;
  end for;
  return SEQ;
end function;

// This two transform a multiset, resp. a tuple, into a sequence.

TypeToSeq:=function(type)
   seq:=[ ]; t:=type;
   while #t ne 0 do Append(~seq, Maximum(t)); 
                    Exclude(~t, Maximum(t));
   end while;
  return seq;
end function;

TupleToSeq:=function(tuple)
  seq:=[];
    for elt in Tuplist(tuple) do
      Append(~seq,elt);
    end for;
  return seq;
end function;

// This script checks if a group has a set of spherical generators of
// the prescribed signature.

ExSphGens:=function(group,type)
  test:=false; seq:=TypeToSeq(type);
  SetCands:=TuplesOfGivenOrders(group,Prune(seq));
  if not IsEmpty(SetCands) then
  for cands in CartesianProduct(SetCands) do
    if Order(&*cands) eq seq[#seq] then
    if sub<group|TupleToSeq(cands)> eq group then
       test:=true; break cands;
    end if; end if;
  end for; end if;
  return test;
end function;

// The next script runs a systematic search on all finite groups and
// produces the list of all triples (basket, pair of signatures, group)
// such that
// 1) the basket is compatible with the input K^2;
// 2) the signatures are compatible with the basket;
// 3) the group has order (K^2+k)/(2*Theta_1*Theta_2) (see 1.11.b)
//    and sets of spherical generators of both signatures.
// If one of the signatures is {*2,3,7*} the group must be perfect, so
// in this case the program first checks if there are perfect groups of
// the right order: if the answer is negative it jumps directly to the
// next case.
// The program skips to check the groups of order bigger than 2000, 1024
// (since there is no complete list avalaible) or of orders in the set
// "badorders" which can be chosen by the user.
// These skipped cases are listed in the second output, and must be
// considered separately.

ListGroups:=function(Ksquare: badorders:={256,512,768,1152,1280,
                                          1536,1728,1792,1920})
  checked:=[* *]; tocheck:=[* *];
  for pair in ListOfTypes(Ksquare) do
    basket:=pair[1]; types:=pair[2]; k:=Invk(basket);
  for pairoftypes in Multisets(types,2) do ord:=(Ksquare+k)/2;
    for T in pairoftypes do ord:=ord/Theta(T);
    end for;
  if IsIntegral(ord) then
  if {*2,3,7*} in pairoftypes and
     NumberOfGroups(PerfectGroupDatabase(),IntegerRing()!ord) eq 0
       then ;
  elif ord gt 2000 or ord in Include(badorders,1024) then
     Append(~tocheck, [* basket, pairoftypes, ord *]);
  else for G in SmallGroups(IntegerRing()!ord: Warning := false) do
    test:=true;
    for T in pairoftypes do
       if not ExSphGens(G,T) then test:=false; break T;
       end if;
    end for;
    if test then Append(~checked, [* basket, pairoftypes, G *]);
    end if; end for;
  end if; end if; end for; end for;
 return checked, tocheck;
end function;


// Each case in the first output of ListGroups(K^2) gives at least a
// surface, but we are interested only in those surfaces having the
// prescribed basket of singularities. The next goal then is to compute
// these singularities.
//
// The next script takes a sequence of elements of a group and a further
// element g and conjugates each element of the sequence with g.

Conjug:=function(seq,elt)
  output:=[];
   for h in seq do Append(~output,h^elt);
   end for;
  return output;
end function;

// The next program computes all possible sets of spherical generators
// of a group of a prescribed signature and returns (to spare memory) only
// one of these sets for each conjugacy class.

SphGenUpToConj:=function(group,type)
  Set:={ }; Rep:={ }; seq:=TypeToSeq(type);
  SetCands:=TuplesOfGivenOrders(group,Prune(seq));
  if not IsEmpty(SetCands) then
  for cands in CartesianProduct(SetCands) do
    if TupleToSeq(cands) notin Set then
    if Order(&*cands) eq seq[#seq] then
    if sub<group|TupleToSeq(cands)> eq group then
       Include(~Rep, Append(TupleToSeq(cands),(&*cands)^-1));
       for g in group do Include(~Set, Conjug(TupleToSeq(cands),g));
       end for;
    end if; end if; end if;
  end for; end if;
  return Rep;
end function;

// Given two sets of spherical generators, the singular points of the
// resulting surface are the image of points in the product of curves
// C_1xC_2 having nontrivial stabilizer. These correspond to pairs
// (g_1,n_1,g_2,n_2) where
// - g_1 is a generator of the first set;
// - g_2 is a generator of the second set;
// 1<=n_1<=ord(g_1); 1<=n_2<=ord(g_2); g_1^n_1=g_2^n_2
// First we write a program which computes the singular points
// coming from a fixed pair (g1,g2).

BasketByAPairOfGens:=function(group,gen1,gen2)
  basket:={* *}; RC:={ }; delta:=GCD(Order(gen1),Order(gen2));
  alpha1:=IntegerRing()!(Order(gen1)/delta);
  alpha2:=IntegerRing()!(Order(gen2)/delta);
  RC2,f2:=RightTransversal(group,sub<group | gen2 >);
  for g2 in RC2 do test:=true;
    for g in sub<group| gen1 > do
    if f2(g2*g) in RC then test:=false; break g;
    end if; end for;
  if test then Include(~RC, g2);
  end if; end for;
  for g in RC do
  for d1 in [1..delta-1] do
  for d2 in [1..delta-1] do
  if (gen1^(d1*alpha1)) eq (gen2^(d2*alpha2))^g then
     Include(~basket,d2/delta); break d1;
  end if; end for; end for; end for;
 return basket;
end function;

// We could use it to compute the basket of singularities of every
// constructed surface, but this is too expensive for our purposes.
// The next program only checks if, given two sets of spherical
// generators and a "candidate" basket, the resulting surface has the
// prescribed basket. The advantage is that in the wrong cases, the
// script stops when it finds a "forbidden" singularities, without
// losing time computing all the other singular points.

CheckSings:=function(basket,gens1,gens2,group)
  test:=true; bas:=basket;
    for gen1 in gens1 do
    for gen2 in gens2 do pb:=BasketByAPairOfGens(group,gen1,gen2);
      for r in pb do r1:=RatNum(Reverse(ContFrac(r)));
      if r in bas then Exclude(~bas,r);
         elif r1 in bas then Exclude(~bas,r1);
         else test:=false; break gen1;
      end if; end for;
    end for; end for;
  return test and IsEmpty(bas);
end function;

// The next script computes all product-quotient surfaces
// with p_g=0, chi=1 and given K^2. It has the same input as ListGroups,
// K^2 and the bad orders (BO), so it does not treat the cases not
// treated by ListGroups, which must be treated separately.

ExistingSurfaces:=function(Ksquare: BO:={256,512,768,1152,1280,
                                          1536,1728,1792,1920})
  M:=[* *];
   for triple in ListGroups(Ksquare: badorders:=BO) do
     basket:=triple[1]; pairsoftypes:=triple[2]; 
     group:=triple[3]; Types:=[]; 
     for type in pairsoftypes do Include(~Types,type); end for;
     SetGens1:=SphGenUpToConj(group,Types[1]);
     if #Types eq 1 then SetGens2:=SetGens1;
                    else SetGens2:=SphGenUpToConj(group,Types[2]);
     end if;
     test:=false;
       for gens1 in SetGens1 do
       for gens2 in SetGens2 do
         if CheckSings(basket,gens1,gens2,group) then test:=true; 
                                                      break gens1;
         end if;
       end for; end for;
       if test then
         Append(~M, [* basket,pairsoftypes,IdentifyGroup(group)*]);
       end if;
     end for;
  return M;
end function;

// We still have not found all possible surfaces. In fact the output of
// ExistingSurfaces(n) gives all possible triples 
// (basket,pair of signatures, group) which give AT LEAST a surface with
// p_g=0 and K^2=n, but there could be more than one. In fact, there are 
// more than one surface for each pair of spherical generators of the
// prescribed types which pass the singularity test, but they are often 
// isomorphic. More precisely, they are isomorphic if the pair of
// spherical generators are equivalent for the equivalence relation 
// generated by Hurwitz moves (on each set of generators separately) 
// and the automorhisms of the group (on both sets simultaneously).
// We need to construct orbits for this equivalence relation.
// The next scripts creates the Automorphism Group of a group as an
// explicit set.

AutGr:=
  function(gr)
    Aut:=AutomorphismGroup(gr); A:={ Aut!1 };
    repeat
      for g1 in Generators(Aut) do
      for g2 in A do
      Include (~A,g1*g2);
      end for; end for;
    until  #A eq #Aut;
  return A;
end function;

// The next script creates the Hurwitz move.

HurwitzMove:=
  function(seq,idx)
  return Insert(Remove(seq,idx),idx+1,seq[idx]^seq[idx+1]);
end function;

// This script, starting from a sequence of elements of a group,
// creates all sequences of elements which are equivalent to the given
// one for the equivalence relation generated by Hurwitz moves,
// and returns (to spare memory) only the ones whose entries have never
// increasing  order.

HurwitzOrbit:=
  function(seq)
  orb:={ }; shortorb:={  }; Trash:={ seq };
    repeat
      ExtractRep(~Trash,~gens); Include(~orb, gens);
        for k in [1..#seq-1] do newgens:=HurwitzMove(gens,k);
        if newgens notin orb then Include(~Trash, newgens);
        end if; end for;
    until IsEmpty(Trash);
    for gens in orb do test:=true;
      for k in [1..#seq-1] do
        if Order(gens[k]) lt Order(gens[k+1]) then test:=false; break k;
        end if;
      end for;
      if test then Include(~shortorb, gens);
      end if;
    end for;
  return shortorb;
end function;

// Now we create all sets of spherical generators of a group of a
// prescribed signature.

SphGens:=function(group,seq)
  Gens:={ }; SetCands:=TuplesOfGivenOrders(group,Prune(seq));
  if not IsEmpty(SetCands) then
  for cands in CartesianProduct(SetCands) do
    if Order(&*cands) eq seq[#seq] then
    if sub<group|TupleToSeq(cands)> eq group then
      Include(~Gens, cands);
    end if; end if;
  end for; end if;
  return Gens;
end function;

// Finally, we can find all surfaces. The next program finds all
// surfaces with a given group, pair of signatures and basket (must be run 
// on the outputs of ExistingSurfaces).

FindSurfaces:=function(basket, pairoftypes, gr)
  Good:={@ @}; Surfaces:={ }; All:={ }; Aut:=AutGr(gr); Types:=[];
  for type in pairoftypes do Append(~Types, type);
  end for;
  seq1:=TypeToSeq(Types[1]); seq2:=TypeToSeq(Types[2]);
  NumberOfCands:=#SphGens(gr,seq1)*#SphGens(gr,seq2);
    for gens1 in SphGens(gr,seq1) do genseq1:=TupleToSeq(gens1);
    for gens2 in SphGens(gr,seq2) do genseq2:=TupleToSeq(gens2);
      if genseq1 cat genseq2 notin All then
        Include(~Surfaces, [Append(genseq1,(&*gens1)^-1),
                            Append(genseq2,(&*gens2)^-1)]);
        orb1:=HurwitzOrbit(Append(genseq1,(&*gens1)^-1));
        orb2:=HurwitzOrbit(Append(genseq2,(&*gens2)^-1));
        for g1 in orb1 do gg1:=Prune(g1);
        for g2 in orb2 do gg2:=Prune(g2);
          if gg1 cat gg2 notin All then
            for phi in Aut do Include(~All, phi(gg1 cat gg2));
            end for;
          end if;
          if #All eq NumberOfCands then break gens1;
          end if;
        end for; end for;
      end if;
    end for; end for;
    for gens in Surfaces do
    if CheckSings(basket,gens[1],gens[2],gr) then
      Include(~Good, gens);
    end if; end for;
  return Good;
end function;

// The next script, FindCurves, uses the same argument of FindSurfaces to find 
// all curves with a given signature and group, modulo Hurwitz moves and inner 
// automorphisms of the group.

FindCurves:=function(type, gr)
  Curves:={ }; All:={ }; seq:=TypeToSeq(type);
  NumberOfCands:=#SphGens(gr,seq); 
    for gens in SphGens(gr,seq) do genseq:=TupleToSeq(gens);
      if genseq notin All then
        Include(~Curves, Append(genseq,(&*gens)^-1));
        orb:=HurwitzOrbit(Append(genseq,(&*gens)^-1));
        for g in orb do gg:=Prune(g);
          if gg notin All then
             for h in gr do Include(~All, Conjug(gg,h));
            end for;
          end if;
          if #All eq NumberOfCands then break gens;
          end if;
        end for;
      end if;
    end for;
  return Curves;
end function;

PolyGroup:=function(seq,gr)
  F:=FreeGroup(#seq); R:={F![1..#seq]};
  for i in [1..#seq] do
    Include(~R,F.i^Order(seq[i]));
  end for;
  P:=quo<F|R>;
  return P, hom<P->gr|seq>;
end function;

DirProd:=function(G1,G2)
  G1xG2:=DirectProduct(G1,G2); vars:=[];
  n:=[NumberOfGenerators(G1),NumberOfGenerators(G2)];
   for i in [1..Wgt(n)] do Append(~vars,G1xG2.i); end for;
   SplittedVars:=Partition(vars,n);
   injs:=[hom< G1->G1xG2 | SplittedVars[1]>,
          hom< G2->G1xG2 | SplittedVars[2]>];
   vars1:=[]; vars2:=[];
   for i in [1..n[1]] do 
     Append(~vars1,G1.i); Append(~vars2,G2!1); 
   end for;
   for i in [1..n[2]] do 
     Append(~vars1,G1!1); Append(~vars2,G2.i); 
   end for;
   projs:=[hom< G1xG2->G1 | vars1>,hom< G1xG2->G2 | vars2>];
  return G1xG2, injs, projs;
end function;

// The next script computes, given two maps A->B (careful, they MUST be  
// between the same groups) the map product induced by the product on B

MapProd:=function(map1,map2)
  seq:=[]; G:=Domain(map1); H:=Codomain(map1);
  if Category(G) eq GrpPC then n:=NPCgens(G); 
  else n:=NumberOfGenerators(G); end if;
  for i in [1..n] do Append(~seq,map1(G.i)*map2(G.i)); end for;
  return hom<G->H|seq>;
end function;

// Finally, this program computes the fundamental group of a product-quotient 
// surface.

Pi1:=function(pairsofseqs,gr)
  T1,f1:=PolyGroup(pairsofseqs[1],gr); 
  T2,f2:=PolyGroup(pairsofseqs[2],gr);
  T1xT2,inT,proT:=DirProd(T1,T2); 
  grxgr,inG:=DirectProduct(gr,gr);
  Diag:=MapProd(inG[1],inG[2])(gr);
  f:=MapProd(proT[1]*f1*inG[1],proT[2]*f2*inG[2]);
  H:=Rewrite(T1xT2,Diag@@f); rels:=[];
  for i in [1..#pairsofseqs[1]] do g1:=pairsofseqs[1][i];
  for j in [1..#pairsofseqs[2]] do g2:=pairsofseqs[2][j];
  for d1 in [1..Order(g1)-1] do
  for d2 in [1..Order(g2)-1] do
    test,h:=IsConjugate(gr,g1^d1,g2^d2);
    if test then for c in Centralizer(gr,g1^d1) do
     Append(~rels, T1xT2.i^d1 * 
            ((T1xT2.(j+#pairsofseqs[1])^d2)^(inT[2]((h^-1*c) @@ f2))));
    end for; end if;
  end for; end for; end for; end for;
  return Simplify(quo<H|rels>);
end function;
\end{verbatim}
\end{footnotesize}

\end{document}